\documentclass[a4paper, 11pt]{article}
\usepackage[T1]{fontenc}
\usepackage[utf8]{inputenc}
\usepackage[all]{xy}
\usepackage{amsfonts}
\usepackage{graphicx}

\usepackage{cite}
\usepackage{enumerate}		
\usepackage{hyperref}		
\usepackage{relsize}
\usepackage{amssymb} 		
\usepackage{mathrsfs} 		
\usepackage{bbm} 		
\usepackage{extarrows} 		
\usepackage{stmaryrd} 		
\usepackage{mathtools}		

\usepackage{amsmath}	
\usepackage{amsthm} 		
\usepackage[nottoc]{tocbibind} 

\usepackage{quiver}  

\theoremstyle{plain}
\newtheorem{thm}{Theorem}
\numberwithin{thm}{section}
\newtheorem{prop}[thm]{Proposition}
\newtheorem{lemma}[thm]{Lemma}

\theoremstyle{definition}
\newtheorem{defn}{Definition}
\numberwithin{defn}{section}

\numberwithin{ex}{section}

\numberwithin{nota}{section}

\numberwithin{remark}{section}

\usepackage{todonotes}
\setuptodonotes{color=blue!10!white}
\usepackage{tikz}
\usepackage{tikz-cd} 
\usetikzlibrary{positioning,intersections} 
\usetikzlibrary{decorations.pathmorphing} 
\usetikzlibrary{arrows}

\newcommand{\yo}{\text{\usefont{U}{min}{m}{n}\symbol{'110}}}
\DeclareFontFamily{U}{min}{}
\DeclareFontShape{U}{min}{m}{n}{<-> dmjhira}{}

\usepackage{xspace} 

\renewcommand{\epsilon}{\varepsilon}
\renewcommand{\phi}{\varphi}

\renewcommand{\lim}{\operatorname{lim}}

\newcommand{\comma}[2]			
{\mbox{$(#1\!\downarrow\!#2)$}}

\newcommand{\Psh}{\mathbf{Psh}}

\newcommand{\dcat}{\mathbb{D}}

\makeatletter
\newbox\xrat@below
\newbox\xrat@above
\newcommand{\xrightarrowtail}[2][]{%
	\setbox\xrat@below=\hbox{\ensuremath{\scriptstyle #1}}%
	\setbox\xrat@above=\hbox{\ensuremath{\scriptstyle #2}}%
	\pgfmathsetlengthmacro{\xrat@len}{max(\wd\xrat@below,\wd\xrat@above)+.6em}%
	\mathrel{\tikz [>->,baseline=-.75ex]
		\draw (0,0) -- node[below=-2pt] {\box\xrat@below}
		node[above=-2pt] {\box\xrat@above}
		(\xrat@len,0) ;}}
\makeatother

\tikzset{Rightarrow/.style={double equal sign distance,>={Implies},->},
	triple/.style={-,preaction={draw,Rightarrow}}}

\begin{document}

\title{Stackification via adjunction}

\author{Wei Zheng \\ \href{}{w1036401560@gmail.com}}

\maketitle

\begin{abstract}
	We establish a form of 2-adjunction (tentatively termed the *fundamental 2-adjunction*), building on the fundamental adjunction proposed by Olivia Caramello and Riccardo Zanfa, which provides a constructive method for the associated stack functor. 
	Additionally, we investigate 2-local homeomorphisms through the lens of indexed fibrations.
\end{abstract}

\tableofcontents

\section{Introduction}

Olivia Caramello and Riccardo Zanfa proposed the fundamental adjunction\cite{caramellozanfa}
\[\begin{tikzcd}
	{\mathbf{Ind}_{\mathcal{C}}} &&& {\mathbf{Topos}^{\mathrm{co}}/\mathbf{Sh}(\mathcal{C},J)}
	\arrow[""{name=0, anchor=center, inner sep=0}, "\Lambda_{\mathbf{Topos}^{\mathrm{co}}/\mathbf{Sh}(\mathcal{C},J)}"', curve={height=18pt}, from=1-1, to=1-4]
	\arrow[""{name=1, anchor=center, inner sep=0}, "\Gamma_{\mathbf{Topos}^{\mathrm{co}}/\mathbf{Sh}(\mathcal{C},J)}"', curve={height=18pt}, from=1-4, to=1-1]
	\arrow["\dashv"{anchor=center, rotate=-90}, draw=none, from=1, to=0]
\end{tikzcd}\]

The classical result $\Lambda\circ\Gamma\cong \iota_J a_J$ is obtained after this adjunction is restricted to the case of presheaves.
\[\begin{tikzcd}
	{\Psh(\cal C)} &&& {\mathbf{Topos}^{\mathrm{co}}/_{1}\mathbf{Sh}(\mathcal{C},J)}
	\arrow[""{name=0, anchor=center, inner sep=0}, "{\Lambda}"', curve={height=18pt}, from=1-1, to=1-4]
	\arrow[""{name=1, anchor=center, inner sep=0}, "{\Gamma}"', curve={height=18pt}, from=1-4, to=1-1]
	\arrow["\dashv"{anchor=center, rotate=-90}, draw=none, from=1, to=0]
\end{tikzcd}\]

We want to know if we have 
\[\Gamma_{\mathbf{Topos}^{\mathrm{co}}/\mathbf{Sh}(\mathcal{C},J)} \circ \Lambda_{\mathbf{Topos}^{\mathrm{co}}/\mathbf{Sh}(\mathcal{C},J)}\simeq i_J\circ s_J\] 
where $s_J$ is denoted to the stackification. At least we want to find all such $\dcat$ that 
\[\Gamma_{\mathbf{Topos}^{\mathrm{co}}/\mathbf{Sh}(\mathcal{C},J)} \circ \Lambda_{\mathbf{Topos}^{\mathrm{co}}/\mathbf{Sh}(\mathcal{C},J)}(\dcat) \simeq i_J\circ s_J (\dcat)\]

We haven't solved the problem yet. But here we get an interesting result. We simply modified the fundamental adjunction (see Theorem \ref{gamma lambda})
\[\begin{tikzcd}
	{\mathbf{Ind}_{\mathcal{C}}} &&& {\textit{2-}\mathbf{Topos}^{\mathrm{co}}/_{2}\mathbf{St}(\mathcal{C},J)}
	\arrow[""{name=0, anchor=center, inner sep=0}, "{\Lambda_{\textit{2-}\mathbf{Topos}^{\mathrm{co}}/_2{}\mathbf{St}(\mathcal{C},J)}}"', curve={height=24pt}, tail reversed, no head, from=1-4, to=1-1]
	\arrow[""{name=1, anchor=center, inner sep=0}, "{\Gamma_{\textit{2-}\mathbf{Topos}^{\mathrm{co}}/_{2}\mathbf{St}(\mathcal{C},J)}}", curve={height=-24pt}, from=1-4, to=1-1]
	\arrow["\dashv"{anchor=center, rotate=-90}, draw=none, from=0, to=1]
\end{tikzcd}\]
Surprisingly, we have this result (Theorem \ref{stackification problem})
\[
	\Gamma_{\textit{2-}\mathbf{Topos}^{\mathrm{co}}/_{2}\mathbf{St}(\mathcal{C},J)}\circ \Lambda_{\textit{2-}\mathbf{Topos}^{\mathrm{co}}/_2{}\mathbf{St}(\mathcal{C},J)}
	\cong i_J\circ s_J
\]

We also proposed an equivalent description of 2-local homeomorphism, see Theorem \ref{slice of 2-topos}

\section{Stackification via double plus construction}
\begin{defn}
	Given a site $(\mathcal{C},J)$
	\begin{enumerate}[(i)]
		\item A descent datum in $\mathbb{D}$ for a $J$-covering sieve $R$ is a pseudocone $\Delta1 \rightarrow \mathbb{D}\circ (\pi_R)^{op}$, 
    		  where $\pi_R:\textstyle \int R\rightarrow C$ is the corresponding discrete fibration associated to $R$.
		\item If $S=\{f_i:\mathrm{\mathrm{dom}}(f_i)\mid i\in I\}$ is a $J$-covering family, 
		then a descent datum in $\mathbb{D}$ for $S$ is a collection data $(U_{f_i},\alpha_{h,k})$ builded by objects $U_{f_i}\in \mathbb{D}(\mathrm{\mathrm{dom}}(f_i))$ for each $i\in I$
		and for any commutative square in $\mathcal{C}$
\[\begin{tikzcd}
	E && {\mathrm{\mathrm{dom}}(f_i)} \\
	\\
	{\mathrm{\mathrm{dom}}(f_j)} && D
	\arrow["k", from=1-1, to=1-3]
	\arrow["h"', from=1-1, to=3-1]
	\arrow["{f_i}", from=1-3, to=3-3]
	\arrow["{f_j}"', from=3-1, to=3-3]
\end{tikzcd}\]
	
    an isomorphim $\alpha_{h,k,j,i}:\mathbb{E}(h)(U_{f_j})\cong \mathbb{E}(k)(U_{f_i})$. 
	If $i=j$, we require that $\alpha_{h,k,j,i}$ is the identity.
	Moreover, all such isomorphisms $\alpha_{h,k,j,i}$ satisfies the coherent conditions: 
	\begin{enumerate}
		\item for any such commutative diagram in $\mathcal{C}$
\[\begin{tikzcd}
	&& {\mathrm{\mathrm{dom}}(f_i)} \\
	\\
	E && {\mathrm{\mathrm{dom}}(f_l)} && D \\
	\\
	&& {\mathrm{\mathrm{dom}}(f_j)}
	\arrow["{f_i}", from=1-3, to=3-5]
	\arrow["k", from=3-1, to=1-3]
	\arrow["h"', from=3-1, to=5-3]
	\arrow["u", from=3-3, to=3-1]
	\arrow["{f_l}", from=3-5, to=3-3]
	\arrow["{f_j}"', from=5-3, to=3-5]
\end{tikzcd}\]
	the isomorphism $\alpha_{h,k,j,i}:\mathbb{E}(h)(U_{f_j})\cong \mathbb{E}(k)(U_{f_i})$ is the composite of 
	$\alpha_{h,u,j,l}:\mathbb{E}(h)(U_{f_j})\cong \mathbb{E}(u)(U_{f_l})$ and $\alpha_{u,k,l,i}:\mathbb{E}(u)(U_{f_l})\cong \mathbb{E}(k)(U_{f_i})$.
		\item If  $i=j$, then the corresponding coherent isomorphism $\alpha_{h,k,j,i}$ is idendity.
	\end{enumerate}
	
	\end{enumerate}
	The morphism between two desecnt datums $(U_{f_i},\alpha_{h,k,j,i})$ and $(V_{f_i},\beta_{h,k,j,i})$ 
	is a collection of morphisms $(\delta_i:U_{f_i}\rightarrow V_{f_i})$ such that for any such commutative square in $\mathcal{C}$
\[\begin{tikzcd}
	E && {\mathrm{\mathrm{dom}}(f_i)} \\
	\\
	{\mathrm{\mathrm{dom}}(f_j)} && D
	\arrow["k", from=1-1, to=1-3]
	\arrow["h"', from=1-1, to=3-1]
	\arrow["{f_i}", from=1-3, to=3-3]
	\arrow["{f_j}"', from=3-1, to=3-3]
\end{tikzcd}\]

We have the following commutative diagram in $\mathbb{D}(E)$

\[\begin{tikzcd}
	{\mathbb{D}(k)(U_{f_i})} && {\mathbb{D}(k)(V_{f_i})} \\
	{\mathbb{D}(h)(U_{f_j})} && {\mathbb{D}(h)(V_{f_j})}
	\arrow["{\mathbb{D}(k)(\delta_{f_i})}", from=1-1, to=1-3]
	\arrow["\sim"{marking, allow upside down}, shift right=2, draw=none, from=1-3, to=2-3]
	\arrow["{\alpha_{h,k,j,i}}", from=2-1, to=1-1]
	\arrow["\sim"{marking, allow upside down}, shift right=2, draw=none, from=2-1, to=1-1]
	\arrow["{\mathbb{D}(h)(\delta_{f_i})}"', from=2-1, to=2-3]
	\arrow["{\beta_{h,k,j,i}}"', from=2-3, to=1-3]
\end{tikzcd}\]
\end{defn}

\begin{thm}[double plus]
    \label{stackification}
	Given a site $(\mathcal{C},J)$. There is a functor $s_J:\mathbf{Fib}_{\mathcal{C}}\rightarrow \mathbf{St}(\mathcal{C},J)$ called the stackification satisfies
\[\begin{tikzcd}
	{\mathbf{Fib}_{\mathcal{C}}} &&& {\mathbf{St}(\mathcal{C},J)}
	\arrow[""{name=0, anchor=center, inner sep=0}, "{i_J}"', from=1-1, to=1-4]
	\arrow[""{name=1, anchor=center, inner sep=0}, "{s_J}"', curve={height=24pt}, from=1-4, to=1-1]
	\arrow["\dashv"{anchor=center, rotate=-96}, draw=none, from=1, to=0]
\end{tikzcd}\]
\end{thm}

\begin{proof}
    Consider an indexed category $\mathbb{D}$.
    We define $\mathcal{G}(\mathbb{D}^{+})$ by 
    \begin{enumerate}[(i)]
        \item objects: $(X,a:R\rightarrow \mathbb{D})$ where $a$ is a descent datum in $\mathbb{D}$ for a $J$-covering sieve $R$.
        \item morphisms: 
        \[(y,\alpha):(Y,b:S\rightarrow \mathbb{D})\rightarrow(X,a:R\rightarrow \mathbb{D})\]
        $y:Y\rightarrow X$ is a morphism in $\mathcal{C}$ and 
        $\alpha:(b:S\rightarrow \mathbb{D})\rightarrow (a\circ y_*):y^*(R)\rightarrow R \rightarrow \mathbb{D})$ 
        is a $\mathcal{C}$-indexed functor $(b\circ i_T:T\rightarrow \mathbb{D})\rightarrow (a\circ y_*\circ i_T:T\rightarrow \mathbb{D})$
        where $T$ is a common refinement of $S$ and $y^*(R)$ in $J(Y)$. 
        Two morphisms $(y,\alpha)$ and $(z,\beta)$ are equal iff $y=z$ and $\alpha$ is locally equals to $\beta$  (i.e. equal on some common refinements).
    \end{enumerate}
    the projection $p_{\mathbb{D}^+}:\mathcal{G}(\mathbb{D}^+)\rightarrow C$ sends $(X,a:R\rightarrow \mathbb{D})$ to $X$.
    
    Thus we have a obvious functor $\mathcal{G}(\eta):\mathcal{G}(\mathbb{D})\rightarrow\mathcal{G}(\mathbb{D}^+)$ 
\[\begin{tikzcd}
	{\mathcal{G}(\mathbb{D})} && {\mathcal{G}(\mathbb{D}^+)} \\
	\\
	& {\mathcal{C}}
	\arrow["{\mathcal{G}(\eta)}", from=1-1, to=1-3]
	\arrow[""{name=0, anchor=center, inner sep=0}, "{p_{\mathbb{D}}}"', from=1-1, to=3-2]
	\arrow[""{name=1, anchor=center, inner sep=0}, "{p_{\mathbb{D}^+}}", from=1-3, to=3-2]
	\arrow[shift left=5, shorten <=7pt, shorten >=7pt, Rightarrow, no head, from=0, to=1]
\end{tikzcd}\]
We note that $\mathbb{D}$ is a $J$-prestack iff $\mathcal{G}(\eta)$ is fully faithful and $\mathbb{D}$ is a $J$-stack iff $\mathcal{G}(\eta)$ is an equivalence.
$\mathbb{D}^+$ is necessarily a $J$-prestack by the very definition of $\mathcal{G}(\mathbb{D}^+)$. 
The functor $\eta:\mathbb{D}\rightarrow \mathbb{D}^+$ is a kind of 2-dimentional universal arrow analoge to the case of $J$-sheaves.

We shall show that $\mathbb{D}^+$ is a $J$-stack if $\mathbb{D}$ is itself a $J$-prestack. Given such diagram
\[\begin{tikzcd}
	R && {\mathbb{D}^+} \\
	\\
	{\yo(X)}
	\arrow["a", from=1-1, to=1-3]
	\arrow["{\textnormal{covering sieve}}"', tail, from=1-1, to=3-1]
\end{tikzcd}\]
The $\mathcal{C}$-indexed functor $a$ sends $[y:Y\rightarrow X]\in R(Y)$ to $a_f\in \mathbb{D}^+(Y)$. 
Note that $a_f$ is a descent datum for $\mathbb{D}$ on a covering $R^f$ (i.e. $R^f\rightarrow \mathbb{D}$)
 and the family $\{a_f\mid f\in R\in J(X)\}$ is compatible by the pseudonaturality.
By the transitivity of Grothendiect topology, we may form the $J$-covering $R*\{R^f\mid f\in R\}$ and a correspondense (in a pseudo sense)
$b:(g\in R^f,f\in R)\mapsto a_{f,g}:=(a_f)_{\mathrm{dom}(g)}([g])$.

We want to show $b:R*\{R^f\mid f\in R\}\rightarrow \mathbb{D}$ is a well-defined descent datum. That is if $f,f'\in R$ and $g\in R^f,g\in R^{f'}$ 
such that $gf=g'f'$ there should be a coherent isomorphism $a_{f,g}\cong a_{f',g'}$ in $\mathbb{D}(\mathrm{dom}(g))$.
It's suffices to check $a_{f,g}$ and $a_{f',g'}$ are locally isomorphic since $\mathbb{D}$ is already a $J$-prestack. 

One could construct $T_{f,h}\subseteq h^*(R^f)\cap R^{fh}$ to be the common refinement of $\mathbb{D}^+(h)(a_f)\cong a_{fh}$ in $\mathbb{D}^+(\mathrm{dom}(h))$
where the coherent isomorphism comes from the pseudonatural constraints of $a:R\rightarrow \mathbb{D}^+$. 
Then $T_{f,g}\cap T_{f',g'}\subseteq R^{fg}=R^{f'g'}$. We have for each $k\in T_{f,g}\cap T_{f',g'}$
\[
    \mathbb{D}(k)(a_{f,g})\cong a_{f,gk}\cong a_{fg,k}= a_{f'g',k}\cong a_{f',g'k} \cong \mathbb{D}(k)(a_{f',g'})
\]
The first isomorphims comes from the pseudonatural constraints of $a_f:R^f\rightarrow \mathbb{D}$. The second comes from the coherent isomorphism of
$\mathbb{D}^+(g)(a_f)\cong a_{fg}$. The third comes from the coherent isomorphism $\mathbb{D}^+(g')(a_f')\cong a_{f'g'}$. 
The fourth is followed by the pseudonatural constraints of $a_{f'}:R^{f'}\rightarrow \mathbb{D}$.

So $b$ is a well-defined descent datum satisfying $\mathbb{D}^+(f)(b)\cong a_f$ for each $f\in R$ in a coherent way.
We have thus shown $\mathcal{G}(\eta):\mathcal{G}(\mathbb{D})\rightarrow \mathcal{G}(\mathbb{D}^+)$ is essential surjective.
Because $\mathcal{G}(\eta)$ is aleady fully faithful by assumption of $\mathbb{D}$, 
we conclude that $p_{\mathbb{D}^+}:\mathcal{G}(\mathbb{D}^+)\rightarrow \mathcal{C}$ is a $J$-stack.

We define $s_J([p_{\mathbb{D}}])$ to be $[p_{\mathbb{D}^{++}}:\mathcal{G}(\mathbb{D}^{++})\rightarrow \mathcal{C}]$.
We conclude that $[p_{\mathbb{D}^{++}}]$ is the 2-dimentional reflection of $[p_{\mathbb{D}}]$ so our thesis follows.

\end{proof}

\section{$\Lambda_{\textit{2-}\mathbf{Topos}^{\mathrm{co}}/_2{}\mathbf{St}(\mathcal{C},J)}\dashv\Gamma_{\textit{2-}\mathbf{Topos}^{\mathrm{co}}/_{2}\mathbf{St}(\mathcal{C},J)}$}\label{chap:relative_sites}

In this section, we gives the explicit constructions of $\Lambda_{\textit{2-}\mathbf{Topos}^{\mathrm{co}}/_2{}\mathbf{St}(\mathcal{C},J)}$ and $\Gamma_{\textit{2-}\mathbf{Topos}^{\mathrm{co}}/_{2}\mathbf{St}(\mathcal{C},J)}$.

\begin{defn}
	Here we gives the explicit constructions of $\Lambda_{\textit{2-}\mathbf{Topos}^{\mathrm{co}}/_2{}\mathbf{St}(\mathcal{C},J)}$ and $\Gamma_{\textit{2-}\mathbf{Topos}^{\mathrm{co}}/_{2}\mathbf{St}(\mathcal{C},J)}$.
	\begin{itemize}
		\item Define the functor $\Lambda_{\textit{2-}\mathbf{Topos}^{\mathrm{co}}/\mathbf{St}(\mathcal{C},J)}$ as
	\[\begin{tikzcd}[sep=scriptsize]
		{\mathbf{Ind}_{\mathcal{C}}} && {\mathbf{Fib}_{\mathcal{C}}} && {\mathrm{Com}_{\mathrm{cont}}/(\mathcal{C},J)} && {\textit{2-}\mathbf{Topos}^{\mathrm{co}}/_{2}\mathbf{St}(\mathcal{C},J)}
		\arrow["{\mathcal{G}}", from=1-1, to=1-3]
		\arrow["{\mathfrak{G}}", from=1-3, to=1-5]
		\arrow["{C_{(-)}^{\mathbf{St}}}", from=1-5, to=1-7]
	\end{tikzcd}\]
		The first of these is the Grothendieck construction
	and the second $p:\mathcal{D}\rightarrow \mathcal{C}$ outputs the result as $p:(\mathcal{D},J_{\mathcal{D}})\rightarrow (\mathcal{C},J)$.
	Finally, the third functor is the canonical one which derive geometric morphisms from comorphisms of sites.
		\item 
		Define the functor
		\[
\Gamma_{\textit{2-}\mathbf{EssTopos}^{\mathrm{co}}/_{2}\mathbf{St}(\mathcal{C},J)}:
\textit{2-}\mathbf{EssTopos}^{\mathrm{co}}/_{2}\mathbf{St}(\mathcal{C},J) \rightarrow \mathbf{Ind}_{\mathcal{C}}
		\]
		as
\[
? \mapsto \textit{2-}\mathbf{EssTopos}^{\mathrm{co}}/_{2}\mathbf{St}(\mathcal{C},J)(\mathbf{St}(\mathcal{C/-},J_{(-)}),?)
\]
	\end{itemize}
\end{defn}

\begin{lemma}
	\label{locally check stack}
	Consider the Giraud's topology $(\mathcal{D},J_{\mathcal{D}})$ for a (Street) fibration $p:\mathcal{D}\rightarrow \mathcal{C}$. 
	An $\mathcal{D}$-indexed category $\mathbb{E}$ is a $J_{\mathcal{D}}$-stack iff $\mathbb{E}\circ (F_{(A,\alpha)})^{op}$ is a $J_X$-stack 
	for every $X$ and every $(A,\alpha)$ in the essential fibre of $p$ at $X$.
	
\end{lemma}
\begin{proof}
	For andy $J_{\mathbb{D}}$-covering $R=\{f_i:\mathrm{dom}(f_i)\rightarrow D\mid i\in I\}$, and assume
	all $f_i$ are cartesian arrows with no lose of generality. Since $\{p(f_i)\mid i\in I\}$ is a $J$-covering family: we can lift them to a 
	$J_{p(D)}$-covering family $S:=\{p(f_i):[p(f_i)]\rightarrow [1_{p(D)}]\mid i\in I\}$ in $\mathcal{C}/p(D)$. 
	\[
	\widehat{R}:=\{p(f_i):[p(f_i)]\rightarrow [1_{p(D)}]\mid i\in I\}
	\]
	We claim that the following diagram commutes up to isomorphism then our thesis follows.
\[\begin{tikzcd}
	{\mathbf{Desc}(M_{D},\mathbb{E})} &&& {\mathbf{Desc}( R,\mathbb{E})} \\
	{\mathbf{Desc}(M_{1_{p(D)}},\mathbb{E}\circ F_{(D,[1_{p(D)}])}^{op})} &&& {\mathbf{Desc}(\widehat{R},\mathbb{E}\circ F_{(D,[1_{p(D)}])}^{op})}
	\arrow["{\mathbf{Desc}( m_R,1)}", from=1-1, to=1-4]
	\arrow[from=1-1, to=2-1]
	\arrow["\sim"{marking, allow upside down}, shift right=2, draw=none, from=1-1, to=2-1]
	\arrow[from=1-4, to=2-4]
	\arrow["\sim"{marking, allow upside down}, shift right=2, draw=none, from=1-4, to=2-4]
	\arrow["{\mathbf{Desc}( m_{\widehat{R}},1)}"', from=2-1, to=2-4]
	\arrow["\sim"{description}, shift left=2, draw=none, from=2-1, to=2-4]
\end{tikzcd}\]

The functor $F_{(D,[1_{p(D)}])}$ sends $g:\mathrm{dom}(g)\rightarrow D$ to $\mathrm{dom}(\widehat{p(g_i)_{D}})$.
$\mathbf{Desc}(M_{D},\mathbb{E}) \simeq \mathbf{Desc}(M_{1_{p(D)}},\mathbb{E}\circ F_{(D,[1_{p(D)}])}^{op})$ is followed by the fibred Yoneda lemma.  
$\mathbf{Desc}(R,\mathbb{E}) \simeq \mathbf{Desc}(\widehat{R},\mathbb{E}\circ F_{(D,[1_{p(D)}])}^{op})$
is due to the properties of cartesian arrows: for each $f_i$, we have the following commutative diagram in $\mathcal{D}$
\[\begin{tikzcd}
	&& {\mathrm{dom}(\widehat{p(f_i)_{D}})} \\
	{\mathrm{dom}(f_i)} \\
	&&& D
	\arrow["{\widehat{p(f_i)_{D}}}", from=1-3, to=3-4]
	\arrow["{\textnormal{\tiny cartesian lifts of $p(f_i)$}}", from=2-1, to=1-3]
	\arrow["\sim"{marking, allow upside down}, shift right=3, draw=none, from=2-1, to=1-3]
	\arrow["{f_i}"', from=2-1, to=3-4]
\end{tikzcd}\]
So the descent datum in $\mathbb{E}$ on $R$ is equivalent to the descent datum in $\mathbb{E}\circ F_{(D,[1_{p(D)}])}^{op}$ on $\widehat{R}$.
\end{proof}

\begin{lemma}
	An indexed functor in $\mathbf{Ind}_{\mathcal{C}}$ is an equivalence if and only if all the 1-components of the indexed functor are equivalences.
\end{lemma}

\begin{thm}
	\label{gamma lambda}
	There is an 2-adjunction
\[\begin{tikzcd}
	{\mathbf{Ind}_{\mathcal{C}}} &&& {\textit{2-}\mathbf{Topos}^{\mathrm{co}}/_{2}\mathbf{St}(\mathcal{C},J)}
	\arrow[""{name=0, anchor=center, inner sep=0}, "{\Lambda_{\textit{2-}\mathbf{Topos}^{\mathrm{co}}/_2{}\mathbf{St}(\mathcal{C},J)}}"', curve={height=24pt}, tail reversed, no head, from=1-4, to=1-1]
	\arrow[""{name=1, anchor=center, inner sep=0}, "{\Gamma_{\textit{2-}\mathbf{Topos}^{\mathrm{co}}/_{2}\mathbf{St}(\mathcal{C},J)}}", curve={height=-24pt}, from=1-4, to=1-1]
	\arrow["\dashv"{anchor=center, rotate=-90}, draw=none, from=0, to=1]
\end{tikzcd}\]
\end{thm} 
\begin{proof}
	Consider an $\mathcal{C}$-indexed category $\mathbb{D}$ with the associated fibration denoted by $p:\mathcal{D}\rightarrow \mathcal{C}$.
	By coYoneda lemma, here we have
	\[
	\mathcal{D}\simeq\mathrm{colim}^{\mathbb{D}}_{ps}C/-
	\]
	We sketch this proposed pseudocolimit diagram in $\mathbf{CAT}$
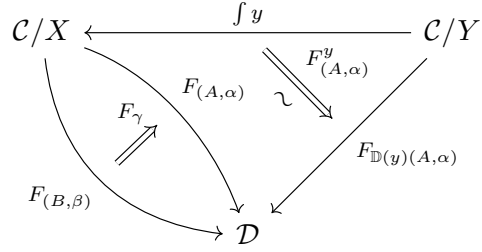
\[\begin{tikzcd}
	{\mathcal{C}/X} &&&& {\mathcal{C}/Y} \\
	\\
	\\
	&& {\mathcal{D}}
	\arrow[""{name=0, anchor=center, inner sep=0}, "{\int y}", tail reversed, no head, from=1-1, to=1-5]
	\arrow[""{name=1, anchor=center, inner sep=0}, "{F_{(A,\alpha)}}"', curve={height=18pt}, tail reversed, no head, from=4-3, to=1-1]
	\arrow[""{name=2, anchor=center, inner sep=0}, "{F_{(B,\beta)}}", curve={height=-30pt}, tail reversed, no head, from=4-3, to=1-1]
	\arrow[""{name=3, anchor=center, inner sep=0}, "{F_{\mathbb{D}(y)(A,\alpha)}}"', tail reversed, no head, from=4-3, to=1-5]
	\arrow["{F_{\gamma}}"'{pos=0.3}, shorten <=8pt, shorten >=8pt, equals, 2tail reversed, from=1, to=2]
	\arrow["{F^{y}_{(A,\alpha)}}"', shorten <=8pt, shorten >=8pt, equals, 2tail reversed, from=3, to=0]
	\arrow["\sim"{marking, allow upside down}, shift left=3, draw=none, from=3, to=0]
\end{tikzcd}\]
	The functor $[(-)^{\mathrm{op}},\mathbf{CAT}]_{ps}:\mathbf{CAT}^{\mathrm{coop}}\rightarrow (\textit{2-}\mathbf{CAT})_{2}$ respects pseudolimit, then
	\[
		[\mathcal{D}^{\mathrm{op}},\mathbf{CAT}]_{ps}\simeq\mathrm{lim}^{\mathbb{D}}_{ps}[(C/-)^{\mathrm{op}},\mathbf{CAT}]_{ps}
	\]
	Here the notation $(\textit{2-}\mathbf{CAT})_{2}$ denotes $\textit{2-}\mathbf{CAT}$ as a 2-dimensional truncation of the 3-category (i.e., removes all nontrivial 3-morphisms and retains only equal 3-morphisms).
	We sketch this proposed pseudolimit diagram in $(\textit{2-}\mathbf{CAT})_{2}$
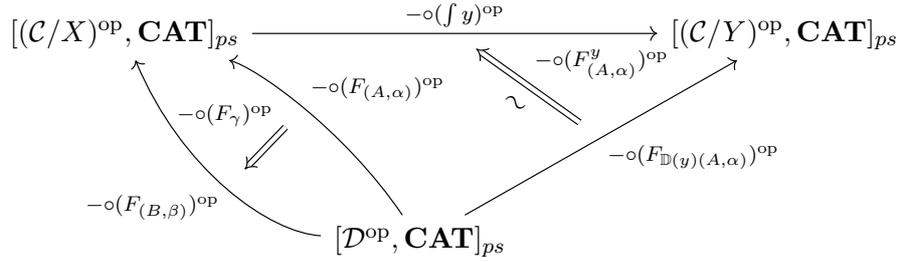
\[\begin{tikzcd}
	{[(\mathcal{C}/X)^{\mathrm{op}},\mathbf{CAT}]_{ps}} &&& {[(\mathcal{C}/Y)^{\mathrm{op}},\mathbf{CAT}]_{ps}} \\
	\\
	\\
	& {[\mathcal{D}^{\mathrm{op}},\mathbf{CAT}]_{ps}}
	\arrow[""{name=0, anchor=center, inner sep=0}, "{-\circ (\int y)^{\mathrm{op}}}", from=1-1, to=1-4]
	\arrow[""{name=1, anchor=center, inner sep=0}, "{-\circ (F_{(A,\alpha)})^{\mathrm{op}}}"', curve={height=18pt}, from=4-2, to=1-1]
	\arrow[""{name=2, anchor=center, inner sep=0}, "{-\circ (F_{(B,\beta)})^{\mathrm{op}}}", curve={height=-30pt}, from=4-2, to=1-1]
	\arrow[""{name=3, anchor=center, inner sep=0}, "{-\circ (F_{\mathbb{D}(y)(A,\alpha)})^{\mathrm{op}}}"', from=4-2, to=1-4]
	\arrow["{-\circ (F_{\gamma})^{\mathrm{op}}}"'{pos=0.3}, shorten <=9pt, shorten >=9pt, Rightarrow, from=1, to=2]
	\arrow["{-\circ (F^{y}_{(A,\alpha)})^{\mathrm{op}}}"', shorten <=9pt, shorten >=9pt, Rightarrow, from=3, to=0]
	\arrow["\sim"{marking, allow upside down}, shift left=3, draw=none, from=3, to=0]
\end{tikzcd}\]
The above diagram remains the proposed pseudolimit cone after restricting to the category of stacks by applying Lemma \ref{locally check stack}
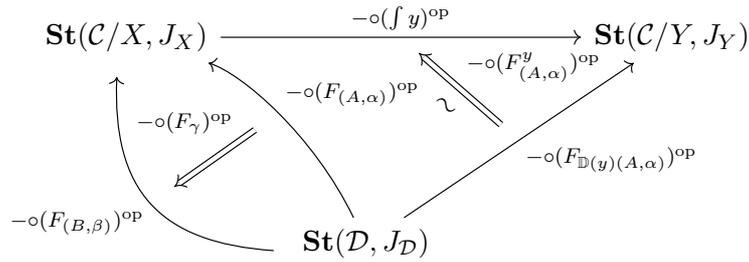
\[\begin{tikzcd}
	{\mathbf{St}(\mathcal{C}/X,J_X)} &&& {\mathbf{St}(\mathcal{C}/Y,J_Y)} \\
	\\
	\\
	& {\mathbf{St}(\mathcal{D},J_{\mathcal{D}})}
	\arrow[""{name=0, anchor=center, inner sep=0}, "{-\circ (\int y)^{\mathrm{op}}}", from=1-1, to=1-4]
	\arrow[""{name=1, anchor=center, inner sep=0}, "{-\circ (F_{(A,\alpha)})^{\mathrm{op}}}"', curve={height=18pt}, from=4-2, to=1-1]
	\arrow[""{name=2, anchor=center, inner sep=0}, "{-\circ (F_{(B,\beta)})^{\mathrm{op}}}", shift left=3, curve={height=-30pt}, from=4-2, to=1-1]
	\arrow[""{name=3, anchor=center, inner sep=0}, "{-\circ (F_{\mathbb{D}(y)(A,\alpha)})^{\mathrm{op}}}"', from=4-2, to=1-4]
	\arrow["{-\circ (F_{\gamma})^{\mathrm{op}}}"'{pos=0.3}, shorten <=10pt, shorten >=10pt, Rightarrow, from=1, to=2]
	\arrow["{-\circ (F^{y}_{(A,\alpha)})^{\mathrm{op}}}"', shorten <=8pt, shorten >=8pt, Rightarrow, from=3, to=0]
	\arrow["\sim"{marking, allow upside down}, shift left=3, draw=none, from=3, to=0]
\end{tikzcd}\]

Any given functor $H:\mathcal{E}\rightarrow \mathbf{St}(\mathcal{D},J_{\mathcal{D}})$ is a inverse image iff
the all the elements in the collection
\[
\{(-\circ F_{(A,\alpha)}^{\mathrm{op}})\circ H \mid \forall (A,\alpha)\}
\]
are inverse images. Necessity: if $H$ is an inverse image, it is known that $(-\circ F_{(A,\alpha)}^{\mathrm{op}})=(C^{\mathbf{St}}_{F_{(A,\alpha)}})^{*}$ is an inverse image.
and thus the composition of two inverse images is an inverse image. Sufficiency: if all $(-\circ F_{(A,\alpha)}^{\mathrm{op}})\circ H$ are inverse images. Since
\[
\{-\circ F_{(A,\alpha)}^{\mathrm{op}} \mid \forall (A,\alpha)\}
\]
jointly reflects equivalences, $H$ respects finite pseudolimits and arbitrary pseudocolimits. So $H$ is an inverse image.
Finally we have the proposed pseudolimit cone in $\textit{2-}\mathbf{Topos}^{\mathrm{co}}/_{2}\mathbf{St}(\mathcal{C},J)$ (the structural arrows of the slices category have been omitted from the diagrams for the sake of abbreviated notation here)
\[\begin{tikzcd}
	{\mathbf{St}(\mathcal{C}/X,J_X)} &&& {\mathbf{St}(\mathcal{C}/Y,J_Y)} \\
	\\
	\\
	& {\mathbf{St}(\mathcal{D},J_{\mathcal{D}})}
	\arrow[""{name=0, anchor=center, inner sep=0}, "{C^{\mathbf{St}}_{F_{(A,\alpha)}}}", curve={height=-18pt}, from=1-1, to=4-2]
	\arrow[""{name=1, anchor=center, inner sep=0}, "{C^{\mathbf{St}}_{F_{(B,\beta)}}}"', shift right=3, curve={height=30pt}, from=1-1, to=4-2]
	\arrow[""{name=2, anchor=center, inner sep=0}, "{C^{\mathbf{St}}_{ \int y}}"', from=1-4, to=1-1]
	\arrow[""{name=3, anchor=center, inner sep=0}, "{C^{\mathbf{St}}_{F_{\mathbb{D}(y)(A,\alpha)}}}", from=1-4, to=4-2]
	\arrow["{C^{\mathbf{St}}_{F_{\gamma}}}"'{pos=0.3}, shorten <=10pt, shorten >=10pt, Rightarrow, from=0, to=1]
	\arrow["{C^{\mathbf{St}}_{F^{y}_{(A,\alpha)}}}"', shorten <=8pt, shorten >=8pt, Rightarrow, from=3, to=2]
	\arrow["\sim"{marking, allow upside down}, shift left=3, draw=none, from=3, to=2]
\end{tikzcd}\]
Thus we prove that
\[
	\mathbf{St}(\mathcal{D},J_{\mathcal{D}}) \simeq \mathrm{colim}^{\mathbb{D}}_{ps}\mathbf{St}(C/-,J_{(-)})
\]
\end{proof}

\section{2-local homeomorphism}

\subsection{The case of indexed categories}
\begin{defn}
	The 2-local homeomorphism is the object in the essential image of $\Lambda_{\textit{2-}\mathbf{Topos}^{\mathrm{co}}/\mathbf{St}(\mathcal{C},J)}$. Its form is equivalent to
\[
	C_{p_{\mathbb{D}}}^{\mathbf{St}} :\mathbf{St}(\mathcal{G}(\mathbb{D}),J_{\mathbb{D}})\rightarrow \mathbf{St}(\mathcal{C},J)
\]
\end{defn}

\begin{defn}
	An $\mathcal{C}$-indexed functor $p:\mathbb{E}\rightarrow \mathbb{D}$ is an $\mathcal{C}$-indexed fibration in the sense that each 1-component of $p$ is a fibration.
And for a morphism $y$ in $\mathcal{C}$, $\mathbb{E}(y)$ outputs the $p^{\mathrm{cod}(y)}$-cartesian arrows as a $p^{\mathrm{dom}(y)}$-cartesian arrows.
\end{defn}
\begin{defn}
$[\mathcal{C}^{op},\mathbf{CAT}]/^{\mathrm{fib}}\mathbb{D}$ is the sub-2-category of $[\mathcal{C}^{op},\mathbf{CAT}]/\mathbb{D}$ whose objects
 are $\cal C$-indexed fibration and morphisms are $\cal C$-indexed morphism of fibrations.	
\end{defn}

\begin{thm}
	Given a site $(\mathcal{C},J)$. There exist an 2-adjoint biequivalence $L_{\mathbb{D}}\dashv R_{\mathbb{D}}$
\[\begin{tikzcd}
	{[\mathcal{C}^{op},\mathbf{CAT}]/^{\mathrm{fib}}\mathbb{D}} && {[\mathcal{G}(\mathbb{D})^{op},\mathbf{CAT}]}
	\arrow["{R_{\mathbb{D}}}", curve={height=-18pt}, from=1-1, to=1-3]
	\arrow["\sim"{description}, draw=none, from=1-1, to=1-3]
	\arrow["{L_{\mathbb{D}}}", curve={height=-18pt}, from=1-3, to=1-1]
\end{tikzcd}\]
\end{thm}

\begin{proof}
	Ideas for constructing $(L_{\mathbb{D}},R_{\mathbb{D}})$:
	\begin{itemize}
		\item Given any $\mathcal{C}$-indexed fibration $p:\mathbb{E}\rightarrow\mathbb{D}$, 
		$R_{\mathbb{D}}([p])$ is a $\mathcal{G}(\mathbb{D})$-indexed category which outputs $(X,U)$ as the essential fiber of $p ^{X}:\mathbb{E}(X)\rightarrow \mathbb{D}(X)$ at $U\in\mathbb{D}(X)$.
		\item Instead, given any $\mathcal{G}(\mathbb{D})$-indexed category $\mathbb{A}$,  
		we have a fibration $$ p_{\mathbb{A}(X,-)}: \mathcal{G}(\mathbb{A}(X,-))\rightarrow \mathbb{D}(X) $$ and it is pseudonatural with respect to the variable $X$. Accordingly, we learn that $L_{\mathbb{D}}$ should be constructed as above.
	\end{itemize}

	The process of proving $(L_{\mathbb{D}}\dashv R_{\mathbb{D}})$ is computationally tedious but straightforward in its thinking. Just verify that
	\[
		[\mathcal{C}^{op},\mathbf{CAT}]/^{\mathrm{fib}}\mathbb{D}(L_{\mathbb{D}}(\mathbb{A}),[p:\mathbb{E}\rightarrow \mathbb{D}])
		\simeq
		[\mathcal{G}(\mathbb{D})^{op},\mathbf{CAT}](\mathbb{A},R_{\mathbb{D}}([p:\mathbb{E}\rightarrow \mathbb{D}]))
	\]
	is pseudonatural in $\mathbb{A}$ and $[p:\mathbb{E}\rightarrow \mathbb{D}]$.

	Given an arrow $(F,\varphi):L_{\mathbb{D}}(\mathbb{A})\rightarrow [p:\mathbb{E}\rightarrow \mathbb{D}]$,
	
	We begin with the following diagram:
\[\begin{tikzcd}[sep=small]
	& {R_{\mathbb{D}}([p])(X,U)} &&&& {\mathbb{E}(X)} \\
	{R_{\mathbb{D}}L_{\mathbb{D}}(\mathbb{A})(X,U)} &&& {L_{\mathbb{D}}(\mathbb{A})(X)=\mathcal{G}(\mathbb{A}(X,-))} \\
	\\
	\\
	{\mathbbm{1}} &&& {\mathbb{D}(X)}
	\arrow[from=1-2, to=1-6]
	\arrow[from=1-2, to=5-1]
	\arrow["\lrcorner"{anchor=center, pos=0.125}, draw=none, from=1-2, to=5-4]
	\arrow[""{name=0, anchor=center, inner sep=0}, "{p^{X}}", from=1-6, to=5-4]
	\arrow["{R_{\mathbb{D}}(F,\varphi)^{(X,U)}}", dashed, from=2-1, to=1-2]
	\arrow[from=2-1, to=2-4]
	\arrow[from=2-1, to=5-1]
	\arrow["{F^{X}}", from=2-4, to=1-6]
	\arrow[""{name=1, anchor=center, inner sep=0}, "{p_{\mathbb{A}(X,-)}}"', from=2-4, to=5-4]
	\arrow[""{name=2, anchor=center, inner sep=0}, "{\mathrm{Ev}_{U}}"', from=5-1, to=5-4]
	\arrow["\varphi"', shorten <=15pt, shorten >=10pt, Rightarrow, from=0, to=1]
	\arrow["\sim"{description}, shift left=2, draw=none, from=0, to=1]
	\arrow["\lrcorner"{anchor=center, pos=0.125}, draw=none, from=2-1, to=2]
\end{tikzcd}\]
The essential fiber at $U$, denoted $R_{\mathbb{D}}L_{\mathbb{D}}(\mathbb{A})(X,-)$, for $p_{\mathbb{A}}(X,-)$ is a Grothendieck fibration, which can be obtained by a strict pullback.

The $R_{\mathbb{D}}L_{\mathbb{D}}(\mathbb{A})(X,U)$ consists of the following information:
\begin{itemize}
	\item objects: $[1_U:U\stackrel{=}{\rightarrow} p_{\mathbb{A}(X,-)}(U,x\in \mathbb{A}(X,U))]$,
	\item morphisms: $(1_U,f):[1_U:U\stackrel{=}{\rightarrow} p_{\mathbb{A}(X,-)}(U,x)]\rightarrow [1_U:U\stackrel{=}{\rightarrow} p_{\mathbb{A}(X,-)}(U,y)]$
	, where $f$ is a morphism $f:x\rightarrow y$ in $\mathbb{A}(X,U)$.
	\item Thus we can consider $R_{\mathbb{D}}L_{\mathbb{D}} (\mathbb{A})(X,U)$ to be the same as $\mathbb{A}(X,U)$.
\end{itemize}

$R_{\mathbb{D}}([p])(X,U)$ as the essential fiber of $p$, consisting of the following information:
\begin{itemize}
	\item objects: $[\alpha: U \stackrel{\sim}{\rightarrow} p^{X}(A)]$
	\item morphisms: $\omega:[\alpha: U \stackrel{\sim}{\rightarrow} p^{X}(A)]\rightarrow [\beta: U \stackrel{\sim}{\rightarrow} p^{X}(B)]$ 
	so that the diagram is commutative in $\mathbb{D}(X)$.
\[\begin{tikzcd}
	&& {p^{X}(A)} \\
	U \\
	&& {p^{X}(B)}
	\arrow["{p^{X}(\omega)}", from=1-3, to=3-3]
	\arrow["\alpha", from=2-1, to=1-3]
	\arrow["\sim"{marking, allow upside down}, shift right=2, draw=none, from=2-1, to=1-3]
	\arrow["\beta"', from=2-1, to=3-3]
	\arrow["\sim"{marking, allow upside down}, shift left=2, draw=none, from=2-1, to=3-3]
\end{tikzcd}\]
\end{itemize}

$R_{\mathbb{D}}$ sends the morphism $(F,\varphi)$ as a morphism $R_{\mathbb{D}}(F,\varphi)$ , with its 1-component $R_{\mathbb{D}}(F,\varphi)^{(X,U)}$. 
sends the object $[1_U:U\stackrel{=}{\rightarrow} p_{\mathbb{A}(X,-)}(U,x)]$ to object
 $[(\varphi^{X})^{-1}(U,x): U\stackrel{\sim}{\rightarrow} p^{X}(F^{X}(U,x))]$
while sending the morphism $(1_U,b)$ to
 $F^{X}(1_U,b): [(\varphi^{X})^{-1}(U,x): U\stackrel{\sim}{\rightarrow} p^{X}(F^{X}(U,x))]\rightarrow [(\varphi^{X})^{-1}(U,y): U\ stackrel{\sim}{\rightarrow} p^{X}(F^{X}(U,y))]$.
We need to verify that $R_{\mathbb{D}}(F,\varphi)^{(X,U)}$ is pseudonatural with respect to the variable $(X,U)$:

Before verifying this, we need to explain in detail how the function $$R_{\mathbb{D}}([p])(y,a):R_{\mathbb{D}}([p])(X,U)\rightarrow R_{\mathbb{D}}([p])(Y,V) $$ works:
It outputs the object $[\alpha:U\stackrel{\sim}{\rightarrow}p^{X}(A)]$ in $R_{\mathbb{D}}([p])(X,U)$ as the object $[\alpha:U\stackrel{\sim}{\rightarrow}p^{X}(A)]$ in $R_{\mathbb{D}}([p])(Y,V)$ 
\[
	[\theta^{p^{Y}}_{p^{y}(A)\mathbb{D}(y)(\alpha)a,\mathbb{E}(y)(A)}:
	V\stackrel{\sim}{\rightarrow}p^{Y}(\mathrm{dom}(\widehat{p^{y}(A)\mathbb{D}(y)(\alpha)a}^{p^{Y}}_{\mathbb{E}(y)(A)}))]
\]
The meaning of the symbolic representation of the upper row can be read off in the following diagram (where the fact that $p^{Y}$ is a fibration is applied).
\[\begin{tikzcd}
	V &&&& {p^{Y}(\mathrm{dom}(\widehat{p^{y}(A)\mathbb{D}(y)(\alpha)a}^{p^{Y}}_{\mathbb{E}(y)(A)}))} \\
	\\
	{\mathbb{D}(y)(U)} && {\mathbb{D}(y)(p^{X}(A))} && {p^{Y}(\mathbb{E}(y)(A))}
	\arrow["{\theta^{p^{Y}}_{p^{y}(A)\mathbb{D}(y)(\alpha)a,\mathbb{E}(y)(A)}}", from=1-1, to=1-5]
	\arrow["\sim"{marking, allow upside down}, shift right=2, draw=none, from=1-1, to=1-5]
	\arrow["a"', from=1-1, to=3-1]
	\arrow["{p^{Y}(\widehat{p^{y}(A)\mathbb{D}(y)(\alpha)a}^{p^{Y}}_{\mathbb{E}(y)(A)})}", from=1-5, to=3-5]
	\arrow["{\mathbb{D}(y)(\alpha)}", from=3-1, to=3-3]
	\arrow["\sim"{marking, allow upside down}, shift right=2, draw=none, from=3-1, to=3-3]
	\arrow["{p^{y}(A)}", from=3-3, to=3-5]
	\arrow["\sim"{marking, allow upside down}, shift right=2, draw=none, from=3-3, to=3-5]
\end{tikzcd}\]

Meanwhile the action of $R_{\mathbb{D}}([p])(y,a)$ on the morphism $\omega:[\alpha]\rightarrow [\beta]$ in $R_{\mathbb{D}}([p])(X,U)$ is determined by the nature of the cartesian arrow, which is canonical and unique, and can be read from the following diagram:
\[\begin{tikzcd}[sep=scriptsize]
	&&& {\scriptscriptstyle p^{Y}(\mathrm{dom}(\widehat{p^{y}(B)\mathbb{D}(y)(\beta)a}^{p^{Y}}_{\mathbb{E}(y)(B)}))} \\
	{\scriptscriptstyle V} && {\scriptscriptstyle p^{Y}(\mathrm{dom}(\widehat{p^{y}(A)\mathbb{D}(y)(\alpha)a}^{p^{Y}}_{\mathbb{E}(y)(A)}))} \\
	\\
	\\
	&& {\scriptscriptstyle \mathbb{D}(y)(p^{X}(B))} & {\scriptscriptstyle p^{Y}(\mathbb{E}(y)(B))} \\
	\\
	{\scriptscriptstyle \mathbb{D}(y)(U)} & {\scriptscriptstyle \mathbb{D}(y)(p^{X}(A))} & {\scriptscriptstyle p^{Y}(\mathbb{E}(y)(A))}
	\arrow["{\scriptscriptstyle p^{Y}(\widehat{p^{y}(B)\mathbb{D}(y)(\beta)a}^{p^{Y}}_{\mathbb{E}(y)(B)})}", from=1-4, to=5-4]
	\arrow["{\scriptscriptstyle \theta^{p^{Y}}_{p^{y}(B)\mathbb{D}(y)(\alpha)a,\mathbb{E}(y)(B)}}", from=2-1, to=1-4]
	\arrow["\sim"{marking, allow upside down}, shift right=2, draw=none, from=2-1, to=1-4]
	\arrow["{\scriptscriptstyle \theta^{p^{Y}}_{p^{y}(A)\mathbb{D}(y)(\alpha)a,\mathbb{E}(y)(A)}}"', from=2-1, to=2-3]
	\arrow["\sim"{marking, allow upside down}, shift left=2, draw=none, from=2-1, to=2-3]
	\arrow["{\scriptscriptstyle a}"', from=2-1, to=7-1]
	\arrow["{\scriptscriptstyle p^{Y}(R_{\mathbb{D}}([p])(y,a)(\omega))}"', dashed, from=2-3, to=1-4]
	\arrow["{\scriptscriptstyle p^{Y}(\widehat{p^{y}(A)\mathbb{D}(y)(\alpha)a}^{p^{Y}}_{\mathbb{E}(y)(A)})}"{pos=0.3}, from=2-3, to=7-3]
	\arrow["{\scriptscriptstyle p^{y}(B)}"{pos=0.3}, from=5-3, to=5-4]
	\arrow["\sim"{marking, allow upside down, pos=0.3}, shift right=2, draw=none, from=5-3, to=5-4]
	\arrow["{\scriptscriptstyle \mathbb{D}(y)(\beta)}", from=7-1, to=5-3]
	\arrow["{\scriptscriptstyle \mathbb{D}(y)(\alpha)}"', from=7-1, to=7-2]
	\arrow["\sim"{marking, allow upside down}, shift left=3, draw=none, from=7-1, to=7-2]
	\arrow["{\scriptscriptstyle \mathbb{D}(y)(p^{X}(\omega))}"', from=7-2, to=5-3]
	\arrow["{\scriptscriptstyle p^{y}(A)}"', from=7-2, to=7-3]
	\arrow["{ \sim}"{marking, allow upside down}, shift left, draw=none, from=7-2, to=7-3]
	\arrow["{\scriptscriptstyle p^{Y}(\mathbb{E}(y)(\omega))}"', from=7-3, to=5-4]
\end{tikzcd}\]

Returning to our task of verifying that $R_{\mathbb{D}}(F,\varphi)^{(X,U)}$ is pseudonaturality in $(X,U)$, we should show that the following diagram is commutative up to isomorphism, which we denote the isomorphism as $R_{\mathbb{D}}(F,\varphi)^{(y,a)}$.
\[\begin{tikzcd}
	{R_{\mathbb{D}}L_{\mathbb{D}}(\mathbb{A})(X,U)} &&& {R_{\mathbb{D}}([p])(X,U)} \\
	\\
	\\
	{R_{\mathbb{D}}L_{\mathbb{D}}(\mathbb{A})(Y,V)} &&& {R_{\mathbb{D}}([p])(Y,V)}
	\arrow["{R_{\mathbb{D}}(F,\varphi)^{(X,U)}}", from=1-1, to=1-4]
	\arrow["{R_{\mathbb{D}}L_{\mathbb{D}}(\mathbb{A})(y,a)}"', from=1-1, to=4-1]
	\arrow["{R_{\mathbb{D}}(F,\varphi)^{(y,a)}}"', shorten <=34pt, shorten >=34pt, Rightarrow, from=1-4, to=4-1]
	\arrow["\sim"{marking, allow upside down}, shift left=2, draw=none, from=1-4, to=4-1]
	\arrow["{R_{\mathbb{D}}([p])(y,a)}"', from=1-4, to=4-4]
	\arrow["{R_{\mathbb{D}}(F,\varphi)^{(Y,V)}}", from=4-1, to=4-4]
\end{tikzcd}\]
Taking any object $[1_U:U\stackrel{=}{\rightarrow}p_{\mathbb{D}}L_{\mathbb{D}}(\mathbb{A})(X,U)]$ in $R_{\mathbb{D}}L_{\mathbb{D}}(\mathbb{A})(X,U)$ , and walking up the road and then walking to the right in the above diagram, we have
\[\scriptstyle
[\theta^{p^{Y}}_{p^{y}(F^{X}(U,x)))\mathbb{D}(y)((\varphi^{X})^{-1}_{(U,X)})a,\mathbb{E}(y)(F^{X}(U,x)))}:
V \stackrel{\sim}{\rightarrow} 
p^{Y}(\mathrm{dom}(\widehat{p^{y}(F^{X}(U,x))\mathbb{D}(y)((\varphi^{X})^{-1}_{(U,X)})a}^{p^{Y}}_{\mathbb{E}(y)(F^{X}(U,x))}))]
\]
What the above symbols indicate can be read off by the upper circuit of the commutative diagram below:
\[\begin{tikzcd}[sep=scriptsize]
	{\scriptscriptstyle V} &&&& {\scriptscriptstyle p^{Y}(\mathrm{dom}(\widehat{p^{y}(F^{X}(U,x))\mathbb{D}(y)((\varphi^{X})^{-1}_{(U,X)})a}^{p^{Y}}_{\mathbb{E}(y)(F^{X}(U,x))}))} \\
	\\
	{\scriptscriptstyle \mathbb{D}(y)(U)} && {\scriptscriptstyle \mathbb{D}(y)(p^{X}(F^{X}(U,x)))} && {\scriptscriptstyle p^{Y}(\mathbb{E}(y)(F^{X}(U,x))))}
	\arrow["{\scriptscriptstyle\theta^{p^{Y}}_{p^{y}(F^{X}(U,x)))\mathbb{D}(y)((\varphi^{X})^{-1}_{(U,X)})a,\mathbb{E}(y)(F^{X}(U,x)))}}", from=1-1, to=1-5]
	\arrow["\sim"{marking, allow upside down}, shift right=2, draw=none, from=1-1, to=1-5]
	\arrow["{\scriptscriptstyle a}"', from=1-1, to=3-1]
	\arrow["{\scriptscriptstyle p^{Y}(\widehat{\widehat{p^{y}(F^{X}(U,x))\mathbb{D}(y)((\varphi^{X})^{-1}_{(U,X)})a}^{p^{Y}}_{\mathbb{E}(y)(F^{X}(U,x))}})}", from=1-5, to=3-5]
	\arrow["{\scriptscriptstyle \mathbb{D}(y)((\varphi^{X})^{-1}_{(U,X)})}", from=3-1, to=3-3]
	\arrow["\sim"{marking, allow upside down}, shift right=2, draw=none, from=3-1, to=3-3]
	\arrow["{\scriptscriptstyle p^{y}(F^{X}(U,x)))}", from=3-3, to=3-5]
	\arrow["\sim"{marking, allow upside down}, shift right=2, draw=none, from=3-3, to=3-5]
\end{tikzcd}\]
On the other hand, by going left and then down the road, we have
\[
[(\varphi^{Y})^{-1}:p_{\mathbb{A}(Y,-)}(V,\mathbb{A}(y,a)(x))\stackrel{\sim}{\rightarrow}p^{Y}F^{Y}(V,\mathbb{A}(y,a)(x))]
\]
The content of this symbol can also be read off by the upper circuit in the following commutative diagrams
\[\begin{tikzcd}
	{\scriptscriptstyle V} & {\scriptscriptstyle p_{\mathbb{A}(Y,-)}(V,\mathbb{A}(y,a)(x))} && {\scriptscriptstyle p^{Y}F^{Y}(V,\mathbb{A}(y,a)(x))} \\
	\\
	{\scriptscriptstyle \mathbb{D}(y)(U)} & {\scriptscriptstyle p_{\mathbb{A}(Y,-)}(\mathbb{D}(y)(U),\mathbb{A}(y,1)(x))} && {\scriptscriptstyle p^{Y}F^{Y}(\mathbb{D}(y)(U),\mathbb{A}(y,1)(x))} && {\scriptscriptstyle p^{Y}(\mathbb{E}(y)F^{X}(U,x))} \\
	& {\scriptscriptstyle p_{\mathbb{A}(Y,-)}(\mathcal{G}(\mathbb{A}(y,-))(U,x))} && {\scriptscriptstyle p^{Y}F^{Y}\mathcal{G}(\mathbb{A}(y,-))(U,x)} \\
	& {\scriptscriptstyle \mathbb{D}(y)(p_{\mathbb{A}(X,-)}(U,x))}
	\arrow["{\scriptscriptstyle 1_V}", from=1-1, to=1-2]
	\arrow["{=}"{marking, allow upside down}, shift right=2, draw=none, from=1-1, to=1-2]
	\arrow["{\scriptscriptstyle a}"', from=1-1, to=3-1]
	\arrow["{\scriptscriptstyle (\varphi^{Y})^{-1}_{(V,\mathbb{A}(y,a)(x))}}", from=1-2, to=1-4]
	\arrow["\sim"{marking, allow upside down}, shift right=2, draw=none, from=1-2, to=1-4]
	\arrow["{\scriptscriptstyle p_{\mathbb{A}(Y,-)}(a,1)}", from=1-2, to=3-2]
	\arrow["{\scriptscriptstyle p^{Y}(F^{Y}(a,1))}"', from=1-4, to=3-4]
	\arrow["{\scriptscriptstyle 1_{\mathbb{D}(y)(U)}}", from=3-1, to=3-2]
	\arrow["{=}"{marking, allow upside down}, shift right=2, draw=none, from=3-1, to=3-2]
	\arrow["{=}"{marking, allow upside down, pos=0.2}, shift right=3, draw=none, from=3-1, to=5-2]
	\arrow["{\scriptscriptstyle \mathbb{D}(y)(1_U)}"', shift right=5, shorten >=14pt, from=3-1, to=5-2]
	\arrow["{\scriptscriptstyle (\varphi^{Y})^{-1}_{(\mathbb{D}(y)(U),\mathbb{A}(y,1)(x))}}", from=3-2, to=3-4]
	\arrow["\sim"{marking, allow upside down}, shift right=2, draw=none, from=3-2, to=3-4]
	\arrow[equals, from=3-2, to=4-2]
	\arrow["{\scriptscriptstyle p^{Y}(F^{y}(U,X))}"', shift right=3, from=3-4, to=3-6]
	\arrow["\sim"{marking, allow upside down}, draw=none, from=3-4, to=3-6]
	\arrow[equals, from=3-4, to=4-4]
	\arrow["{\scriptscriptstyle (\varphi^{Y})^{-1}_{(\mathcal{G}(\mathbb{A}(y,-))(U,x))}}", from=4-2, to=4-4]
	\arrow["\sim"{marking, allow upside down}, shift right=2, draw=none, from=4-2, to=4-4]
	\arrow[equals, from=4-2, to=5-2]
\end{tikzcd}\]
Now we wish to prove that the two objects $[(\varphi^{Y})^{-1}]$ 
and $[\theta^{p^{Y}}_{p^{y}(F^{X}(U,x)))\mathbb{D}(y)((\varphi^{X})^{-1}_{(U,X)})a,\mathbb{E}(y)(F^{X}(U,x)))}]$ are isomorphic in $R_{\mathbb{D}}([p])(Y,V)$.  Observe that
\[
	p^{Y}(F^{y}(U,X))\circ(\varphi^{Y})^{-1}_{\mathcal{G}(\mathbb{A}(y,-))(U,x)} 
	= p^{y}(F^{X}(U,x)))\circ \mathbb{D}(y)((\varphi^{X})^{-1}_{(U,X)})
\]
The above equation is read in the 2-dimensional commutative diagram below:
\[\begin{tikzcd}
	{\mathbbm{1}} \\
	\\
	& {\mathcal{G}(\mathbb{A}(X,-))} &&& {\mathcal{G}(\mathbb{A}(Y,-))} \\
	&&& {\mathbb{E}(X)} &&& {\mathbb{E}(Y)} \\
	\\
	&&& {\mathbb{D}(X)} &&& {\mathbb{D}(Y)}
	\arrow["{\mathrm{Ev}_{(U,x)}}"', from=1-1, to=3-2]
	\arrow["{\mathcal{G}(\mathbb{A}(y,-))}", from=3-2, to=3-5]
	\arrow["{F^{X}}"', from=3-2, to=4-4]
	\arrow[""{name=0, anchor=center, inner sep=0}, "{p_{\mathbb{A}(X,-)}}"', from=3-2, to=6-4]
	\arrow["{F^{y}}"', Rightarrow, from=3-5, to=4-4]
	\arrow["\sim"{marking, allow upside down}, shift left=3, draw=none, from=3-5, to=4-4]
	\arrow["{F^{Y}}", from=3-5, to=4-7]
	\arrow[""{name=1, anchor=center, inner sep=0}, "{p_{\mathbb{A}(Y,-)}}"'{pos=0.7}, dashed, from=3-5, to=6-7]
	\arrow["{\mathbb{E}(y)}"', from=4-4, to=4-7]
	\arrow["{p^{X}}"', from=4-4, to=6-4]
	\arrow["{p^{Y}}"', from=4-7, to=6-7]
	\arrow["{p^{y}}", shorten <=30pt, shorten >=30pt, Rightarrow, from=6-4, to=4-7]
	\arrow["\sim"{marking, allow upside down}, shift right=3, draw=none, from=6-4, to=4-7]
	\arrow["{\mathbb{D}(y)}"', from=6-4, to=6-7]
	\arrow["{\varphi^{X}}"', shorten >=5pt, Rightarrow, from=4-4, to=0]
	\arrow["\sim"{marking, allow upside down}, shift left=3, draw=none, from=4-4, to=0]
	\arrow["{\varphi^{Y}}", shift left, shorten >=5pt, Rightarrow, dashed, from=4-7, to=1]
	\arrow["\sim"{marking, allow upside down}, shift right, draw=none, from=4-7, to=1]
\end{tikzcd}\]
Since $[(\varphi^{Y})^{-1}]$ is a cartesian lift of $(\varphi^{Y})^{-1}_{\mathcal{G}(\mathbb{A}(y,-))(U,x)}$ along $p^{Y}$ and
\[
	[\theta^{p^{Y}}_{p^{y}(F^{X}(U,x)))\mathbb{D}(y)((\varphi^{X})^{-1}_{(U,X)})a,\mathbb{E}(y)(F^{X}(U,x)))}]
\] 
is a cartesian lift of $p^{y}(F^{X}(U,x)))\circ \mathbb{D}(y)((\varphi^{X})^{-1}_{(U,X)})$ along $p^{Y}$.
Moreover the isomorphism $F^{y}(U,X)$ is automatically a $p^{Y}$-cartesian arrow.
Finally, by the property of cartesian arrows, we have that the two objects are isomorphic in $R_{\mathbb{D}}([p])(Y,V)$. Thus we complete the verification that $R_{\mathbb{D}}(F,\varphi)^{(X,U)}$ is pseudonatural in the according arguments.

We construct the unit $\eta_{\mathbb{A}}:\mathbb{A}\rightarrow R_{\mathbb{D}}L_{\mathbb{D}}(\mathbb{A})$ .
For each $(X,U)$ in $\mathcal{G}(\mathbb{D})$, the 1-component of $\eta_{\mathbb{A}}$ is defined as
\[
	\eta_{\mathbb{A}}^{(X,U)}:\mathbb{A}(X,U)\rightarrow R_{\mathbb{D}}L_{\mathbb{D}}(\mathbb{A})(X,U)
\]
It outputs the object $x\in \mathbb{A}(X,U)$ as $[1_U:U\stackrel{=}{\rightarrow} p_{\mathbb{A}(X,-)}(U,x)]\in R_{\mathbb{D}}L_{\mathbb{D}}(\mathbb{A})(X ,U)$ 
Also output the morphism $f:x\rightarrow y$ in $\mathbb{A}(X,U)$ as $(1_U,f)$ in $R_{\mathbb{D}}L_{\mathbb{D}}(\mathbb{A})(X,U)$.
The pseudonaturality of $\eta_{\mathbb{A}}^{(X,U)}$ is clearly obtained (in fact, $\eta_{\mathbb{A}}^{(X,U)}$ is strictly natural).

We define the functor $(F^{\sharp})^{(X,U)}:\mathbb{A}(X,U)\rightarrow R_{\mathbb{D}}([p])(X,U)$ 
as the compostion of $\eta_{\mathbb{A}}^{(X,U)}:\mathbb{A}(X,U)\rightarrow R_{\mathbb{D}}L_{\mathbb{D}}(\mathbb{A})(X,U)$ and $R_{\mathbb{D}}(F,\varphi)^{(X,U)}$.
It outputs the object $x \in \mathbb{A}(X,U)$ as $[(\varphi^{X})^{-1}(U,x): U\stackrel{\sim}{\rightarrow} p^{X}(F^{X}(U,x))]$
and sends $f\in \mathbb{A}(X,U)$ to $F^{X}(1_U,f)$.
Finally we need to check the pseudonaturality of $(F^{\sharp})^{(X,U)}:\mathbb{A}(X,U)\rightarrow R_{\mathbb{D}}([p])(X,U)$ with respect to $(X,U)\in\mathcal{G}(\mathbb{D})$,
That is, check that for each $(y,a):(X,U)\rightarrow (Y,V)$ in $\mathcal{G}(\mathbb{D})$, there is the following commutative diagram.
\[\begin{tikzcd}
	{\mathbb{A}(X,U)} &&& {R_{\mathbb{D}}([p])(X,U)} \\
	\\
	{\mathbb{A}(Y,V)} &&& {R_{\mathbb{D}}([p])(Y,V)}
	\arrow["{(F^{\sharp})^{(X,U)}}", from=1-1, to=1-4]
	\arrow["{\mathbb{A}(y,a)}"', from=1-1, to=3-1]
	\arrow[shorten <=27pt, shorten >=27pt, Rightarrow, from=1-4, to=3-1]
	\arrow["\sim"{marking, allow upside down}, shift right=2, draw=none, from=1-4, to=3-1]
	\arrow["{R_{\mathbb{D}}([p])(y,a)}", from=1-4, to=3-4]
	\arrow["{(F^{\sharp})^{(Y,V)}}"', from=3-1, to=3-4]
\end{tikzcd}\]
But that commutative diagram is already mentioned in the pseudonaturality of $\eta_{\mathbb{A}}^{(X,U)}$ and $R_{\mathbb{D}}(F,\varphi)^{(X,U)}$.

Instead, from a $\mathcal{G}(\mathbb{D})$-indexed functor $H:\mathbb{A}\rightarrow R_{\mathbb{D}}([p:\mathbb{E}\rightarrow \mathbb{D}])$.
We have $L_{\mathbb{D}}(H):L_{\mathbb{D}}(\mathbb{A})\rightarrow L_{\mathbb{D}}(R_{\mathbb{D}}([p]))$, whose 1-component has the form
$$
\mathcal{G}(H^{(X,-)}):(U,x\in \mathbb{A}(X,U))\mapsto 
(U,[H^{(X,U)}(x):U\stackrel{\sim}{\rightarrow}p(A_{x})]\in \mathcal{G}(R_{\mathbb{D}}([p])(X,-)))
$$
We have the counit $\epsilon_{[p]}$, whose 1-component is an obvious projection:
\[\begin{tikzcd}
	{L_{\mathbb{D}}(\mathbb{A})(X)} && {L_{\mathbb{D}}(R_{\mathbb{D}}([p]))(X)} \\
	{\mathcal{G}(\mathbb{A}(X,-))} && {\mathcal{G}(R_{\mathbb{D}}([p])(X,-))} && {\mathbb{E}(X)} \\
	\\
	\\
	&& {\mathbb{D}(X)}
	\arrow[equals, from=1-1, to=2-1]
	\arrow[equals, from=1-3, to=2-3]
	\arrow["{\mathcal{G}(H^{(X,-)})}", from=2-1, to=2-3]
	\arrow[""{name=0, anchor=center, inner sep=0}, "{p_{\mathbb{A}(X,-)}}"', from=2-1, to=5-3]
	\arrow["{\epsilon_{[p]}^{X}}", from=2-3, to=2-5]
	\arrow[""{name=1, anchor=center, inner sep=0}, "{p_{R_{\mathbb{D}}([p])(X,-)}}", from=2-3, to=5-3]
	\arrow[""{name=2, anchor=center, inner sep=0}, "{p^{X}}", from=2-5, to=5-3]
	\arrow[shift left=5, shorten <=19pt, shorten >=19pt, equals, from=0, to=1]
	\arrow["\sim", shift left=5, shorten <=17pt, shorten >=17pt, equals, 2tail reversed, from=1, to=2]
\end{tikzcd}\]
It is clear that $\mathcal{G}(H^{(X,-)})$ and $\epsilon_{[p]}^{X}$ is pseudonatural in $X$.
So we define $(H^{\flat})^{X}$ as the composition of $\mathcal{G}(H^{(X,-)})$ and $\epsilon_{[p]}^{X}$, which is pseudonatual in $X$.

Return to the following correspondence
\[
		[\mathcal{C}^{op},\mathbf{CAT}]/^{\mathrm{fib}}\mathbb{D}(L_{\mathbb{D}}(\mathbb{A}),[p:\mathbb{E}\rightarrow \mathbb{D}])
		\leftrightarrow
		[\mathcal{G}(\mathbb{D})^{op},\mathbf{CAT}](\mathbb{A},R_{\mathbb{D}}([p:\mathbb{E}\rightarrow \mathbb{D}]))
\]
The $(-)^{\sharp}$ and $(-)^{\flat}$ defined above are quasi-inverse to each other and they are pseudonatural with respect to the variables $\mathbb{A}$ and $[p]$.
Thus we get the 2-adjunction $(L_{\mathbb{D}}\dashv R_{\mathbb{D}})$.
The unit and counit are equivalences.
We conclude that $(L_{\mathbb{D}}\dashv R_{\mathbb{D}})$ is an 2-adjoint biequivalence.
\end{proof}

\subsection{The case of stacks}

\begin{thm}
	\label{slice of 2-topos}
	Given a site $(\mathcal{C},J)$, we have the following commutative diagram up to isomorphism
\[\begin{tikzcd}
	{[\mathcal{C}^{op},\mathbf{CAT}]/^{\mathrm{fib}}\mathbb{D}} &&& {[\mathcal{G}(\mathbb{D})^{op},\mathbf{CAT}]} \\
	\\
	\\
	{\mathbf{St}(\mathcal{C},J)/^{\mathrm{fib}}s_J(\mathbb{D})} &&& {\mathbf{St}(\mathcal{G}(\mathbb{D}),J_{\mathbb{D}})}
	\arrow[""{name=0, anchor=center, inner sep=0}, "{R_{\mathbb{D}}}", curve={height=-12pt}, from=1-1, to=1-4]
	\arrow[""{name=1, anchor=center, inner sep=0}, "{s_J/\mathbb{D}}"', curve={height=18pt}, from=1-1, to=4-1]
	\arrow[""{name=2, anchor=center, inner sep=0}, "{L_{\mathbb{D}}}", curve={height=-12pt}, from=1-4, to=1-1]
	\arrow[""{name=3, anchor=center, inner sep=0}, "{s_{J_{\mathbb{D}}}}"', curve={height=18pt}, from=1-4, to=4-4]
	\arrow[""{name=4, anchor=center, inner sep=0}, "{i_J/\mathbb{D}:=\eta_{\mathbb{D}}^*}"', from=4-1, to=1-1]
	\arrow[""{name=5, anchor=center, inner sep=0}, "{R_{\mathbb{D}}^J}", curve={height=-12pt}, from=4-1, to=4-4]
	\arrow[""{name=6, anchor=center, inner sep=0}, "{i_{J_{\mathbb{D}}}}"', from=4-4, to=1-4]
	\arrow[""{name=7, anchor=center, inner sep=0}, "{L_{\mathbb{D}}^J}", curve={height=-12pt}, from=4-4, to=4-1]
	\arrow["\dashv"{anchor=center}, draw=none, from=1, to=4]
	\arrow["\dashv"{anchor=center}, draw=none, from=3, to=6]
	\arrow["\dashv"{anchor=center, rotate=90}, draw=none, from=2, to=0]
	\arrow["\dashv"{anchor=center, rotate=90}, draw=none, from=7, to=5]
\end{tikzcd}\]
	where $L_{\mathbb{D}}^J \dashv R_{\mathbb{D}}^J$ is an 2-adjoint biequivalence.
\end{thm}

\begin{proof}
	We divide the proof into the following steps:
	\begin{enumerate}[(i)]
		\item $s_J/\mathbb{D}$ sends the objects in $[\mathcal{C}^{op},\mathbf{CAT}]/^{\mathrm{fib}}\mathbb{D}$ to
		the objects in $\mathbf{St}(\mathcal{C},J)/^{\mathrm{fib}}s_J(\mathbb{D})$.
		\item The objects in the essential image of $R_{\mathbb{D}}^J:=R_{\mathbb{D}}\circ i_J/\mathbb{D}$ are $J_{\mathbb{D}}$-stacks.
		\item Define $L_{\mathbb{D}}^J$ as $s_J/\mathbb{D}\circ L_{\mathbb{D}} \circ i_{J_{\mathbb{D}}}$ , so $(L_{\mathbb{D}}^J\dashv R_\mathbb{D}^J)$.
		\item $(L_{\mathbb{D}}^J\dashv R_\mathbb{D}^J)$ is an 2-adjoint biequivalence.
	\end{enumerate}

 Proof of (i): Consider an $\mathcal{C}$-indexed fibration $p:\mathbb{E}\rightarrow\mathbb{D}$, we shall prove that for any $X\in \mathcal{C}$,
$s_J(p^{X}):s_J(\mathbb{E})(X)\rightarrow s_J(\mathbb{D})(X)$ is a fibraton.
We are going to prove that $(p^{X})^{+}:\mathbb{E}^{+}(X)\rightarrow \mathbb{D}^{+}(X)$ is a fibration since $s_J=(-)^{++}$.
For any morphism $x:X\rightarrow(p^{X})^{+}(A)$ in $\mathbb{D}^{+}(X)$, assume $X=(X_{f},\alpha_{f,g})_{f\in R}$,
$A=(A_{h},\beta_{h,k})_{h\in S}$ , $x=(x_t:X_{t}\rightarrow p^{\mathrm{dom}(t)}(A_{t}))_{t\in T\subseteq R\cap S}$.
For each $t\in T$, we have
\[\begin{tikzcd}
	{p^{\mathrm{dom}(t)}(\mathrm{dom}(\widehat{x_t}^{p^{\mathrm{dom}(t)}}_{A_{t}}))} \\
	\\
	{X_{t}} &&& {p^{\mathrm{dom}(t)}(A_{t})}
	\arrow["{p^{\mathrm{dom}(t)}(\widehat{x_t}^{p^{\mathrm{dom}(t)}}_{A_{t}})}", from=1-1, to=3-4]
	\arrow["{\theta^{p^{\mathrm{dom}(t)}}_{x_t,A_t}}", from=3-1, to=1-1]
	\arrow["{\sim }"{marking, allow upside down}, shift right=2, draw=none, from=3-1, to=1-1]
	\arrow["{x_t}", from=3-1, to=3-4]
\end{tikzcd}\]
and for each morphism $s$ that composable with $t$ , there are
\[\begin{tikzcd}[sep=small]
	&& {\scriptscriptstyle \mathbb{D}(s)(p^{\mathrm{dom}(t)}(\mathrm{dom}(\widehat{x_t}^{p^{\mathrm{dom}(t)}}_{A_{t}})))} \\
	& {\scriptscriptstyle p^{\mathrm{dom}(s)}(\mathbb{E}(s)(\mathrm{dom}(\widehat{x_t}^{p^{\mathrm{dom}(t)}}_{A_{t}})))} \\
	{\scriptscriptstyle p^{\mathrm{dom}(s)}(\mathrm{dom}(\widehat{x_{ts}}^{p^{\mathrm{dom}(s)}}_{A_{ts}}))} && {\scriptscriptstyle \mathbb{D}(s)(X_{t})} & {\scriptscriptstyle \mathbb{D}(s)(p^{\mathrm{dom}(t)}(A_{t}))} \\
	\\
	{\scriptscriptstyle X_{ts}} && {\scriptscriptstyle p^{\mathrm{dom}(s)}(\mathbb{E}(s)(A_{t}))} \\
	& {\scriptscriptstyle p^{\mathrm{dom}(s)}(A_{ts})}
	\arrow["{\scriptscriptstyle p^{s}(\mathrm{dom}(\widehat{x_t}^{p^{\mathrm{dom}(t)}}_{A_{t}}))}"'{pos=0.9}, from=1-3, to=2-2]
	\arrow["\sim"{marking, allow upside down}, shift left=3, draw=none, from=1-3, to=2-2]
	\arrow["{\scriptscriptstyle \mathbb{D}(s)(p^{\mathrm{dom}(t)}(\widehat{x_t}^{p^{\mathrm{dom}(t)}}_{A_{t}}))}"{pos=0.7}, from=1-3, to=3-4]
	\arrow["{\scriptscriptstyle p^{\mathrm{dom}(s)}(\gamma_{s,t})}"'{pos=0.7}, from=2-2, to=3-1]
	\arrow["\sim"{description}, shift left=3, draw=none, from=2-2, to=3-1]
	\arrow["{\scriptscriptstyle p^{\mathrm{dom}(s)}(\mathbb{E}(s)(\widehat{x_t}^{p^{\mathrm{dom}(t)}}_{A_{t}}))}"{description, pos=0.2}, from=2-2, to=5-3]
	\arrow["{\scriptscriptstyle p^{\mathrm{dom}(s)}(\widehat{x_{ts}}^{p^{\mathrm{dom}(s)}}_{A_{ts}})}"{pos=0.9}, from=3-1, to=6-2]
	\arrow["{\scriptscriptstyle \mathbb{D}(s)(\theta^{p^{\mathrm{dom}(t)}}_{x_t,A_t})}", from=3-3, to=1-3]
	\arrow["{\sim }"{marking, allow upside down}, shift right=2, draw=none, from=3-3, to=1-3]
	\arrow["{\scriptscriptstyle \mathbb{D}(s)(x_t)}", from=3-3, to=3-4]
	\arrow["{\scriptscriptstyle \alpha_{s,t}}"{description, pos=0.4}, from=3-3, to=5-1]
	\arrow["\sim"{marking, allow upside down}, shift right=2, draw=none, from=3-3, to=5-1]
	\arrow["{\scriptscriptstyle p^{s}(A_t)}", from=3-4, to=5-3]
	\arrow["\sim"{marking, allow upside down}, shift right=3, draw=none, from=3-4, to=5-3]
	\arrow["{\scriptscriptstyle \theta^{p^{\mathrm{dom}(s)}}_{x_{ts},A_{ts}}}", from=5-1, to=3-1]
	\arrow["\sim"{marking, allow upside down}, shift right=2, draw=none, from=5-1, to=3-1]
	\arrow["{\scriptscriptstyle x_{ts}}"', from=5-1, to=6-2]
	\arrow["{\scriptscriptstyle p^{\mathrm{dom}(s)}(\beta_{s,t})}", from=5-3, to=6-2]
	\arrow["\sim"{marking, allow upside down}, shift right=2, draw=none, from=5-3, to=6-2]
\end{tikzcd}\]
So there is a descent datum $(\mathrm{dom}(\widehat{x_{t}}^{p^{\mathrm{dom}(s)}}_{A_{t}}),\gamma_{s,t})_{t\in T}$ in $\mathbb{E}$.
So we have the following commutative diagram in $\mathbb{D}^{+}(X)$
\[\begin{tikzcd}
	{(p^{X})^{+}((\mathrm{dom}(\widehat{x_{t}}^{p^{\mathrm{dom}(s)}}_{A_{t}}),\gamma_{s,t})_{t\in T})} \\
	\\
	{(X_{f},\alpha_{f,g})_{f\in R}} &&& {(p^{+})^{X}((A_{h},\beta_{h,k})_{h\in S})}
	\arrow["{(p^{+})^{X}((x_t:X_{t}\rightarrow p^{\mathrm{dom}(t)}(A_{t}))_{t\in T\subseteq R\cap S})}", from=1-1, to=3-4]
	\arrow["{(\theta^{p^{\mathrm{dom}(t)}}_{x_t,A_t})_{t\in T\subseteq R\cap S}}", from=3-1, to=1-1]
	\arrow["{\sim }"{marking, allow upside down}, shift right=2, draw=none, from=3-1, to=1-1]
	\arrow["{(x_t:X_{t}\rightarrow p^{\mathrm{dom}(t)}(A_{t}))_{t\in T\subseteq R\cap S}}", from=3-1, to=3-4]
\end{tikzcd}\]
and $(x_t:X_{t}\rightarrow p^{\mathrm{dom}(t)}(A_{t}))_{t\in T\subseteq R\cap S}$ 
is a $(p^{+})^{X}$-cartesian lift of $(x_t:X_{t}\rightarrow p^{\mathrm{dom}(t)}(A_{t}))_{t\in T\subseteq R\cap S}$ in $\mathbb{E}^{+}(X)$.

Proof of (ii): From the object $s_J(p):s_J(\mathbb{E})\rightarrow s_J(\mathbb{D})$ in $\mathbf{St}(\mathcal{C},J)/^{\mathrm{fib}}s_J(\mathbb{D})$,
$i_J/\mathbb{D}$ sends it to  $q:\mathbb{H}\rightarrow \mathbb{D}$, that is the pesudopullback of $s_J(p)$ along $\eta_{\mathbb{D}}$.
Then $R_\mathbb{D}$ sends $q:\mathbb{H}\rightarrow \mathbb{D}$ to $\mathbb{G}:\mathcal{G}(\mathbb{D})^{op}\rightarrow \mathbf{CAT}$.
The above process is shown below:
\[\begin{tikzcd}
	{\mathbb{G}(X,U)} &&& {\mathbb{H}(X)} &&& {s_J(\mathbb{E})(X)} \\
	\\
	{\mathbbm{1}} &&& {\mathbb{D}(X)} &&& {s_J(\mathbb{D})(X)}
	\arrow[from=1-1, to=1-4]
	\arrow[from=1-1, to=3-1]
	\arrow["\lrcorner"{anchor=center, pos=0.125}, draw=none, from=1-1, to=3-4]
	\arrow[from=1-4, to=1-7]
	\arrow["{q^{X}}"', from=1-4, to=3-4]
	\arrow["\lrcorner"{anchor=center, pos=0.125}, draw=none, from=1-4, to=3-7]
	\arrow["{s_J(p)^{X}}"', from=1-7, to=3-7]
	\arrow["{\textnormal{fibration}}"'{pos=0.7}, draw=none, from=1-7, to=3-7]
	\arrow[shorten <=34pt, shorten >=26pt, Rightarrow, from=3-1, to=1-4]
	\arrow["\sim"{marking, allow upside down}, shift left=3, draw=none, from=3-1, to=1-4]
	\arrow["{\mathrm{Ev}_{U}}"', from=3-1, to=3-4]
	\arrow[shorten <=33pt, shorten >=25pt, Rightarrow, from=3-4, to=1-7]
	\arrow["\sim"{marking, allow upside down}, shift left=3, draw=none, from=3-4, to=1-7]
	\arrow["{\eta_{\mathbb{D}}^{X}}", from=3-4, to=3-7]
\end{tikzcd}\]
We need to prove that $\mathbb{G}$ is a $J_{\mathbb{D}}$-stack. First, we show that every descent datum in $\mathbb{G}$ is effective.
That is to say, the descent datum in $\mathbb{G}$ 
$$([\lambda_i:\eta_{\mathbb{D}}(\mathbb{D}(f_i)(U))\stackrel{\sim}{\rightarrow}s_J(p)^{\mathrm{dom}(f_i)}(A^i,\alpha^i)])_{i\in I}$$ 
on the $J_{\mathbb{D}}$-covering family
$$\{(f_i,1):(\mathrm{dom}(f_i),\mathbb{D}(f_i)(U))\rightarrow(X,U)\mid i\in I\}$$
is effective.
For each $i$, $(A^i,\alpha^i)$ is shorthand for $(A^i_{f},\alpha^i_{h,k,f,g})_{f\in R}$, where $A^i_f$ denotes the descent datum in a $\mathbb{E}$.
Thus $(A^i,\alpha^i)$ is an object in $s_J(\mathbb{E})$.
According to the definition of the Giraud topology $J_{\mathbb{D}}$ , the  $\{f_i:\mathrm{dom}(f_i)\rightarrow X\mid i\in I\}$ is a $J$-covering.
So the descent datum $(A^i,\alpha^i)_{i\in I}$ in that $s_J(\mathbb{E})$ has a gluing $(B,\beta)\in s_J(\mathbb{E})(X)$.
So here exists a gluing 
$$[\lambda:\eta_{\mathbb{D}}(U)\stackrel{\sim}{\rightarrow}s_J(p)^{X}(B,\beta)]$$ 
for the descent datum in $\mathbb{G}$
$$[\lambda_i:\eta_{\mathbb{D}}(\mathbb{D}(f_i)(U))\stackrel{\sim}{\rightarrow}s_J(p)^{\mathrm{dom}(f_i)}(A^i,\alpha^i)]_{i\in I}$$  

Consider a morphism between the above types of descent datums.
 $$\scriptstyle (x):[\lambda_i:\eta_{\mathbb{D}}(\mathbb{D}(f_i)(U))\stackrel{\sim}{\rightarrow}s_J(p)^{\mathrm{dom}(f_i)}(A^i,\alpha^i)]_{i\in I}\rightarrow 
[\sigma_i:\eta_{\mathbb{D}}(\mathbb{D}(f_i)(U)) \stackrel{\sim}{\rightarrow}s_J(p)^{\mathrm{dom}(f_i)}(C^i,\gamma^i)]_{i\in I} $$ 
Assuming that they each exist gluings: $[\lambda:\eta_{\mathbb{D}}(U)\stackrel{\sim}{\rightarrow}s_J(p)^{X}(B,\beta)]$  
and $[\sigma:\eta_{\mathbb{D}}(U)\stackrel{\sim}{\rightarrow}s_J(p)^{X}(D,\delta)]$.
Then the morphism $(x)$ between the descent datums has a unique guling $z:(B,\beta)\rightarrow (D,\delta)$, since $s_J(\mathbb{E})$ is a $J$-stack itself.
So we prove that $\mathbb{G}$ is a $J_{\mathbb{D}}$-stack.

Proof of (iii): Just verify directly.

Proof of (iv): We will show that both $L_{\mathbb{D}}^J\circ R_\mathbb{D}^J$ and $R_{\mathbb{D}}^J\circ L_\mathbb{D}^J$ are equivalent to the identity morphisms.
Since $s_J/\mathbb{D}\circ i_J/\mathbb{D}\simeq 1_{\mathbf{St}(\mathcal{C},J)/^{\mathrm{fib}}s_J(\mathbb{D})}$ , so we have
$L_{\mathbb{D}}^J\circ R_\mathbb{D}^J\simeq 1_{\mathbf{St}(\mathcal{C},J)/^{\mathrm{fib}}s_J(\mathbb{D})}$.
And for the proof of $R_{\mathbb{D}}^J\circ L_\mathbb{D}^J\simeq 1_{\mathbf{St}(\mathcal{G}(\mathbb{D}),J_{\mathbb{D}})}$,
Consider a $J_{\mathbb{D}}$-stack $\mathbb{A}:(\mathcal{G}(\mathbb{D}))^{op}\rightarrow \mathbf{CAT}$, We first compute the following pseudopullback:
\[\begin{tikzcd}
	{\mathbb{G}(X,U)} && {\mathbb{H}(X)} && {(s_J(L_{\mathbb{D}}(\mathbb{A})))(X)=s_J\mathcal{G}(\mathbb{A}(X,-))} \\
	\\
	{\mathbbm{1}} && {\mathbb{D}(X)} && {(s_J\mathbb{D})(X)}
	\arrow[from=1-1, to=1-3]
	\arrow[from=1-1, to=3-1]
	\arrow["\lrcorner"{anchor=center, pos=0.125}, draw=none, from=1-1, to=3-3]
	\arrow[from=1-3, to=1-5]
	\arrow[from=1-3, to=3-3]
	\arrow["\lrcorner"{anchor=center, pos=0.125}, draw=none, from=1-3, to=3-5]
	\arrow["{s_Jp_{\mathbb{A}(X,-)}}", from=1-5, to=3-5]
	\arrow[shorten <=19pt, shorten >=19pt, Rightarrow, from=3-1, to=1-3]
	\arrow["\sim"{marking, allow upside down}, shift left=3, draw=none, from=3-1, to=1-3]
	\arrow["{\mathrm{Ev}_{U}}", from=3-1, to=3-3]
	\arrow[shorten <=27pt, shorten >=27pt, Rightarrow, from=3-3, to=1-5]
	\arrow["\sim"{marking, allow upside down}, shift left=3, draw=none, from=3-3, to=1-5]
	\arrow["{\eta_{\mathbb{D}}^{X}}", from=3-3, to=3-5]
\end{tikzcd}\]
We want to show that there is an equivalence:
$$\mathbb{G}(X,U)\simeq \mathbb{A}(X,U)
$$ and this equivalence is pseudonatural with respect to the variable $(X,U)$.  
Take any object in $s_J\mathcal{G}(\mathbb{A}(X,-))$
$
(A_{h},\alpha_{h})_{{h}\in H}
$
assume that
$$ [\theta:\eta_{\mathbb{D}}^{X}(U)\stackrel{\sim}{\rightarrow}s_Jp_{\mathbb{A}(X,-)}(A_{h},\alpha_{h})_{{h}\in H}] $$ 
is an object in $\mathbb{G}(X,U)$,
And the information of these descent datums $A_h$ are described by the following diagrams
\[\begin{tikzcd}
	{(U_k,x_k\in\mathbb{A}(X,U_k))} \\
	{\mathcal{G}(\mathbb{A}(\mathrm{dom}(f_k),-))} \\
	\\
	&& {\mathcal{G}(\mathbb{A}(\mathrm{dom}(h),-))} \\
	{\mathcal{G}(\mathbb{A}(\mathrm{dom}(f_j),-))} \\
	{(U_j,x_j\in\mathbb{A}(X,U_j))}
	\arrow["{\textnormal{is an object of}}", squiggly, from=1-1, to=2-1]
	\arrow[""{name=0, anchor=center, inner sep=0}, "{\mathcal{G}(\mathbb{A}(f_k,-))}"', from=4-3, to=2-1]
	\arrow[""{name=1, anchor=center, inner sep=0}, "{\mathcal{G}(\mathbb{A}(f_j,-))}", from=4-3, to=5-1]
	\arrow["{\textnormal{is an object of}}"', squiggly, from=6-1, to=5-1]
	\arrow["{\textnormal{$J$-covering $R_i$}}"', shift right=5, curve={height=12pt}, shorten <=7pt, shorten >=7pt, dashed, no head, from=0, to=1]
\end{tikzcd}\]
According to the isomorphism $\theta$,  The descent datum consisting of these $U_k$ has a gluing $\mathbb{D}(h)(U)$ in $\mathbb{D}(\mathrm{dom}(h))$.
So we get the following diagram
\[\begin{tikzcd}
	{x_k} \\
	{\mathbb{A}(\mathrm{dom}(f_k),U_k)} \\
	\\
	&& {\mathbb{A}(\mathrm{dom}(h),\mathbb{D}(h)(U))} \\
	{\mathbb{A}(\mathrm{dom}(f_j),U_j)} \\
	{x_j}
	\arrow["{\textnormal{is an object of}}", squiggly, from=1-1, to=2-1]
	\arrow[""{name=0, anchor=center, inner sep=0}, "{\mathbb{A}(f_k,1)}"', from=4-3, to=2-1]
	\arrow[""{name=1, anchor=center, inner sep=0}, "{\mathbb{A}(f_j,1)}", from=4-3, to=5-1]
	\arrow["{\textnormal{is an object of}}"', squiggly, from=6-1, to=5-1]
	\arrow["{\textnormal{$J_{\mathbb{D}}$-covering family}}"', shift right=5, curve={height=12pt}, shorten <=7pt, shorten >=7pt, dashed, no head, from=0, to=1]
\end{tikzcd}\]
Further we have
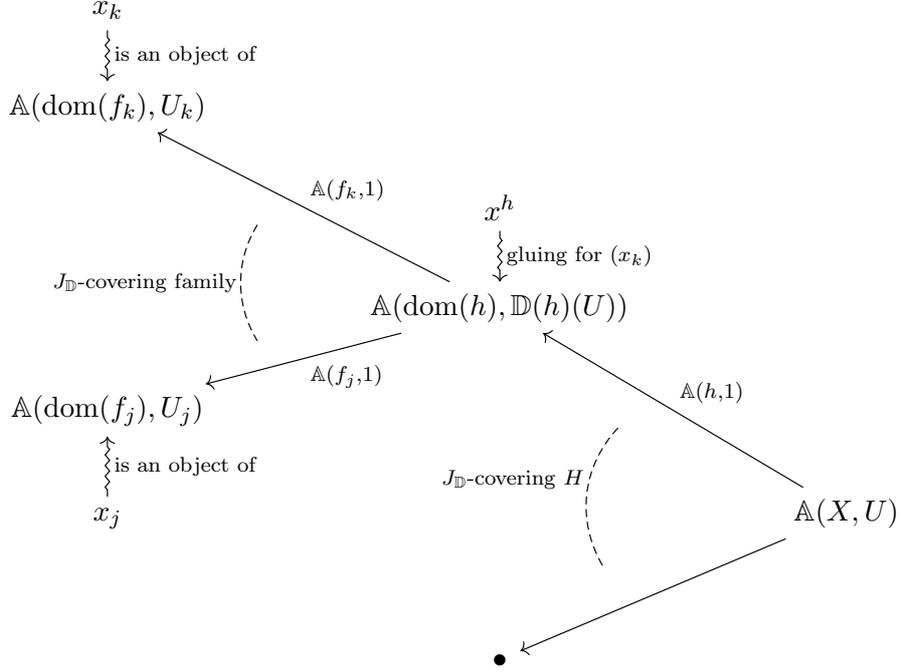
\[\begin{tikzcd}
	{x_k} \\
	{\mathbb{A}(\mathrm{dom}(f_k),U_k)} \\
	&& {x^{h}} \\
	&& {\mathbb{A}(\mathrm{dom}(h),\mathbb{D}(h)(U))} \\
	{\mathbb{A}(\mathrm{dom}(f_j),U_j)} \\
	{x_j} &&&& {\mathbb{A}(X,U)} \\
	\\
	&& \bullet
	\arrow["{\textnormal{is an object of}}", squiggly, from=1-1, to=2-1]
	\arrow["{\textnormal{gluing for $(x_k)$}}", squiggly, from=3-3, to=4-3]
	\arrow[""{name=0, anchor=center, inner sep=0}, "{\mathbb{A}(f_k,1)}"', from=4-3, to=2-1]
	\arrow[""{name=1, anchor=center, inner sep=0}, "{\mathbb{A}(f_j,1)}", from=4-3, to=5-1]
	\arrow["{\textnormal{is an object of}}"', squiggly, from=6-1, to=5-1]
	\arrow[""{name=2, anchor=center, inner sep=0}, "{\mathbb{A}(h,1)}"', from=6-5, to=4-3]
	\arrow[""{name=3, anchor=center, inner sep=0}, from=6-5, to=8-3]
	\arrow["{\textnormal{$J_{\mathbb{D}}$-covering family}}"', shift right=5, curve={height=12pt}, shorten <=7pt, shorten >=7pt, dashed, no head, from=0, to=1]
	\arrow["{\textnormal{$J_{\mathbb{D}}$-covering $H$}}"', shift right=5, curve={height=18pt}, shorten <=10pt, shorten >=10pt, dashed, no head, from=2, to=3]
\end{tikzcd}\]
Note that $\mathbb{A}$ is a $J_{\mathbb{D}}$-stack.  
The descent datums in the above graph consisting of all such $x^{h}$ with $h$ as indicator has a gluing $x$ in $\mathbb{A}(X,U)$.

Based on the process discussed above, we have obtained the natural equivalence $\mathbb{G}(X,U)\simeq \mathbb{A}(X,U)$.
\end{proof}

\begin{prop}
	\label{keq prop}
	Given a site $(\mathcal{C},J)$, there is an equivalence
	\[
	\mathbf{St}(\mathcal{C},J)\simeq \textit{2-}\mathbf{Topos}/_{2}^{et}\mathbf{St}(\mathcal{C},J)
	\]
	where the 2-category $\textit{2-}\mathbf{Topos}/_{2}^{et}\mathbf{St}(\mathcal{C},J)$ is the full subcategory
	of $\textit{2-}\mathbf{Topos}/_{2}\mathbf{St}(\mathcal{C},J)$ spanned by the 2-local homeomorphisms.
\end{prop}

\begin{proof}
	Define
	\[
	\Phi:\mathbf{St}(\mathcal{C},J)\rightarrow  \textit{2-}\mathbf{Topo}/_{2}^{et}\mathbf{St}(\mathcal{C},J)
	\]
	It sends $J$-stack $\mathbb{D}$ to
	\[
	\Pi_{\mathbb{D}}:\mathbf{St}(\mathcal{C},J)/^{\mathrm{fib}}\mathbb{D} \rightarrow \mathbf{St}(\mathcal{C},J)
	\] 
	Here the definition of
	$\Pi_{\mathbb{D}}$ is
	\[
		\mathbf{St}(\mathcal{C},J)/^{\mathrm{fib}}\mathbb{D} \stackrel{\sim}{\rightarrow} \mathbf{St}(\mathcal{D},J_{D}) \stackrel{C^{\mathbf{St}}_{p}}{\rightarrow} \mathbf{St}(\mathcal{C},J)
	\]
	We construct the quasi-inverse of $\Phi$:
	\[
	\Psi:\textit{2-}\mathbf{Topo}/_{2}^{et}\mathbf{St}(\mathcal{C},J)\rightarrow \mathbf{St}(\mathcal{C},J)
	\]
	It sends
	\[
		\Pi_{\mathbb{D}}:\mathbf{St}(\mathcal{C},J)/^{\mathrm{fib}}\mathbb{D} \rightarrow \mathbf{St}(\mathcal{C},J)
	\]
	to $(C^{\mathbf{St}}_{p})_{!}(\Delta \mathbbm{1})$ ,where $p$ is the fibration associated to the indexed category $\mathbb{D}$.
\end{proof}

\section{Stackification via adjunction}

\begin{thm}
	\label{stackification problem}
	Recall the adjunction mentioned in Theorem \ref{gamma lambda},
	we have
\[
	\Gamma_{\textit{2-}\mathbf{Topos}^{\mathrm{co}}/_{2}\mathbf{St}(\mathcal{C},J)}\circ \Lambda_{\textit{2-}\mathbf{Topos}^{\mathrm{co}}/_2{}\mathbf{St}(\mathcal{C},J)}
	\cong i_J\circ s_J
\]
and the adjunciton
$\Lambda_{\textit{2-}\mathbf{Topos}^{\mathrm{co}}/_2{}\mathbf{St}(\mathcal{C},J)}\dashv\Gamma_{\textit{2-}\mathbf{Topos}^{\mathrm{co}}/_{2}\mathbf{St}(\mathcal{C},J)}$
can be restrict to an adjoint biequivalence $\Lambda^{'}\dashv\Gamma^{'}$:
\[\begin{tikzcd}
	{\mathbf{Ind}_{\mathcal{C}}} &&& {\textit{2-}\mathbf{Topos}^{\mathrm{co}}/_{2}\mathbf{St}(\mathcal{C},J)} \\
	\\
	\\
	{\mathbf{St}(\mathcal{C},J)} &&& {\textit{2-}\mathbf{Topos}^{\mathrm{co}}/^{et}_{2}\mathbf{St}(\mathcal{C},J)}
	\arrow[""{name=0, anchor=center, inner sep=0}, "{\Lambda_{\textit{2-}\mathbf{Topos}^{\mathrm{co}}/_2{}\mathbf{St}(\mathcal{C},J)}}"', curve={height=24pt}, tail reversed, no head, from=1-4, to=1-1]
	\arrow[""{name=1, anchor=center, inner sep=0}, "{\Gamma_{\textit{2-}\mathbf{Topos}^{\mathrm{co}}/_{2}\mathbf{St}(\mathcal{C},J)}}", curve={height=-24pt}, from=1-4, to=1-1]
	\arrow[hook', from=4-1, to=1-1]
	\arrow["{\Lambda^{'}}", curve={height=-24pt}, from=4-1, to=4-4]
	\arrow["\sim"{description}, draw=none, from=4-1, to=4-4]
	\arrow[hook', from=4-4, to=1-4]
	\arrow["{\Gamma^{'}}", curve={height=-24pt}, from=4-4, to=4-1]
	\arrow["\dashv"{anchor=center, rotate=-90}, draw=none, from=0, to=1]
\end{tikzcd}\]
\end{thm}

\begin{proof}
	For any indexed category $\mathbb{D}$, we have
\[
	\Gamma_{\textit{2-}\mathbf{Topos}^{\mathrm{co}}/_{2}\mathbf{St}(\mathcal{C},J)}
	\circ \Lambda_{\textit{2-}\mathbf{Topos}^{\mathrm{co}}/_2{}\mathbf{St}(\mathcal{C},J)}(\mathbb{D})
	=\textit{2-}\mathbf{Topos}^{\mathrm{co}}/_{2}\mathbf{St}(\mathcal{C},J)(\Pi_{s_J\yo(-)},\Pi_{\mathbb{D}})
\]
	apply Theorem \ref{slice of 2-topos} and Proposition \ref{keq prop}, here is
\[
	\textit{2-}\mathbf{Topos}^{\mathrm{co}}/_{2}\mathbf{St}(\mathcal{C},J)(\Pi_{s_J\yo(-)},\Pi_{s_J\mathbb{D}})
	\simeq
	\mathbf{St}(\mathcal{C},J)(s_J\yo(-),s_J\mathbb{D})
\]
	Again, by Yoneda's Lemma, we have
\[
	\Gamma_{\textit{2-}\mathbf{Topos}^{\mathrm{co}}/_{2}\mathbf{St}(\mathcal{C},J)}\circ \Lambda_{\textit{2-}\mathbf{Topos}^{\mathrm{co}}/_2{}\mathbf{St}(\mathcal{C},J)}
	\cong i_J\circ s_J
\]
Then we have $(\Lambda^{'}\dashv \Gamma^{'})$ is an adjoint biequivalence.
\end{proof}

\bibliography{biblio}{}
\bibliographystyle{abbrv}

\end{document}